\documentclass[11pt,a4paper,onesided]{article}

\usepackage{ amssymb, amsmath, amsthm, graphics, xspace, enumerate}
\usepackage{graphicx}
\usepackage[a4paper,colorlinks=true,citecolor=blue,urlcolor=blue,linkcolor=blue,bookmarksopen=true,unicode=true,pdffitwindow=true]{hyperref}
 \hypersetup{pdfauthor={Nikola Sandri\'{c}}}
\hypersetup{pdftitle={Ergodicity of L\'evy-Type Processes}}

\newtheorem{tm}{tm}[section]
\newtheorem{theorem}[tm]{Theorem}

\newtheorem{corollary}[tm]{Corollary}
\newtheorem{proposition}[tm]{Proposition}

\newtheorem{definition}[tm]{Definition}

\newtheorem{example}[tm]{Example}

\newcommand {\R} {\ensuremath{\mathbb{R}}}

\newcommand {\N} {\ensuremath{\mathbb{N}}}
\newcommand {\CC} {\ensuremath{\mathbb{C}}}
\newcommand{\process}[1]{\{#1_t\}_{t\geq0}}

\numberwithin{equation}{section}
\def\be{\begin{equation}}
\def\ee{\end{equation}}
\begin{document}

 \title{Ergodicity of L\'evy-Type Processes}
 \author{Nikola Sandri\'{c}\\
Institut f\"ur Mathematische Stochastik\\ Fachrichtung Mathematik, Technische Universit\"at Dresden, 01062 Dresden, Germany\\
and\\
Department of Mathematics\\
         Faculty of Civil Engineering, University of Zagreb, 10000 Zagreb,
         Croatia \\
        Email: nsandric@grad.hr }

 \maketitle
\begin{center}
{
\medskip

} \end{center}

\begin{abstract}
In this paper,  conditions for transience, recurrence, ergodicity and strong, subexponential (polynomial) and exponential ergodicity of a class of Feller processes   are derived. The conditions are given in terms of the coefficients of the corresponding infinitesimal generator. As a consequence,  mixing properties of these processes are also discussed.

\end{abstract}

 \smallskip


%
%
%
%


\section{Introduction}
The main goal of this paper is to derive conditions for transience, recurrence, ergodicity and strong, subexponential (polynomial) and exponential ergodicity of
Feller processes generated by an integro-differential operator of the form
\begin{align}\label{eq1.1}\mathcal{L}f(x)&=-a(x)f(x)+\sum_{i=1}^{d}b_i(x)\frac{\partial f(x)}{\partial x_i}+\frac{1}{2}\sum_{i,j=1}^{d}c_{ij}(x)\frac{\partial^{2}f(x)}{\partial x_i\partial x_j}\nonumber
\\&\ \ \ +\int_{\R^{d}}\left(f(y+x)-f(x)-\sum_{i=1}^{d}y_i\frac{\partial f(x)}{\partial x_i}1_{B(0,1)}(y)\right)\nu(x,dy).\end{align}  The class of Feller processes of this type is known as L\'evy-type processes (see Section \ref{s2} for details).
This work  is  motivated
by  the works of P. Mandl \cite{mandl} and R. N. Bhattacharya \cite{bhat-rec, bhat-cor} (see also \cite{friedman} and \cite{friedman-book}), where the authors obtained sufficient conditions for  transience, recurrence and strong ergodicity of  conservative elliptic diffusion processes, that is, processes governed by an  operator of the form \eqref{eq1.1} but with $a(x)=0$ and $\nu(x,dy)=0$ for all $x\in\R^{d}$.
More precisely, under certain regularity conditions of the coefficients $b(x):=(b_i(x))_{1\leq i\leq d}$ and $c(x):=(c_{ij}(x))_{1\leq i,j\leq d}$ (local boundedness of $b(x)$ and continuity, symmetry and nonsingularity of $c(x)$), by defining $A(x):=(1/2)|x|^{-2}\sum_{i=1}^{d}c_{ii}(x),$ $B(x):=|x|^{-2}\sum_{i=1}^{d}x_ib_i(x)$, $C(x):=|x|^{-4}\sum_{i,j=1}^{d}x_ix_jc_{ij}(x),$   $\underline{I}(r):=\inf_{|x|=r}(2A(x)-C(x)+2B(x))/C(x)$, $\overline{I}(r):=\sup_{|x|=r}(2A(x)-C(x)+2B(x))/C(x)$,
\begin{align}\label{eq1.2}&
T(r):=\int_{r_0}^{r}\exp\left\{-\int_{r_0}^{s}\underline{I}(u)/u\, du\right\}ds, \quad R(r):=\int_{r_0}^{r}\exp\left\{-\int_{r_0}^{s}\overline{I}(u)/u\, du\right\}ds, \nonumber\\  &E(r):=\int_{r_0}^{r}\left(\exp\left\{-\int_{r_0}^{s}\overline{I}(u)\Big{/}u\, du\right\}\int_s^{\infty}\exp\left\{\int_{r_0}^{u}\overline{I}(v)\Big{/}v\, dv\right\}\Big{/}\inf_{|x|=u}C(x)\,du\right)ds,
\end{align}
they have shown the following:
 \begin{enumerate}
   \item [(i)] the underlying  process is transient if for some $r_0>0$, $\lim_{r\longrightarrow\infty}T(r)<\infty$;
   \item [(ii)] the underlying  process is recurrent if for some $r_0>0$, $\lim_{r\longrightarrow\infty}R(r)=\infty$;
   \item [(iii)] the underlying process is strongly ergodic if  for some $r_0>0$, $\lim_{r\longrightarrow\infty}R(r)=\infty$ and $E(r)<\infty$ for all $r\geq r_0$.
 \end{enumerate}
The so-called Lyapunov functions $T(r)$, $R(r)$ and $E(r)$, defined in \eqref{eq1.2}, appear as an appropriate optimization of solutions of certain second-order ordinary differential equations associated to $\mathcal{L}$ (see \cite{bhat-rec} for details).
By using a similar approach, in the general situation,  certain ordinary integro-differential equations  are associated  to the operator $\mathcal{L}$. However,  to the best of our knowledge, it is not completely clear how to solve these equations. Therefore, we construct ``universal" Lyapunov functions which do not depend on the coefficients of $\mathcal{L}$ and share some properties of  $T(r)$, $R(r)$ and $E(r)$. By considering the simplest elliptic diffusion case: $a(x)=0$, $b(x)=0$ and $c(x)=I$ for all $x\in\R^{d}$ (here, $I$ denotes the $d\times d$-identity matrix), that is, the case of a standard $d$-dimensional Brownian motion, it is easy to see that  adequate choices are   $1-r^{-\alpha}$ (for transience) and $\ln r$ or $r^{\alpha}$ (for recurrence) for some $\alpha>0$ and all $r>0$ large enough (see also \cite{sandric-spa}, \cite{sandric-rectrans}, \cite{sandric-ergodic}, \cite{stramer-tweedie-sinica} and \cite{wang-ergodic}). Then, by using these functions and following the ideas presented in \cite{bhat-rec}, we are in a position to derive the desired conditions (see Theorem \ref{tm1.3}).

Except
for elliptic diffusions, whose  transience, recurrence and ergodicity property
 has been studied in  \cite{bhat-rec},  \cite{friedman}, \cite{friedman-book}, \cite{mandl} and \cite{stramer-tweedie-sinica},   transience, recurrence and ergodicity  of certain special cases of L\'evy-type processes only have already been
considered in the literature.
More precisely, the transience and recurrence of
L\'evy processes have been studied extensively in \cite{sato-book}.
In \cite{vere2} and \cite{vere1}  the author has studied mixing properties of elliptic diffusions, and in \cite{douc} and \cite{subgeometric} conditions for the polynomial ergodicity of elliptic diffusions and compound Poisson-process driven Ornstein-Uhlenbeck-type processes have been obtained. The  transience, recurrence, strong  ergodicity and mixing properties of general Ornstein-Uhlenbeck-type processes have been studied in \cite{masuda}, \cite{watanabe}, \cite{sato} and \cite{shiga}.  In the closely related paper \cite{wang-ergodic} the author has  discussed the recurrence and strong ergodicity of one-dimensional L\'evy-type processes, while in \cite{wee1} and \cite{wee2} the  transience, recurrence and strong ergodicity of  multidimensional L\'evy-type processes but with uniformly bounded jumps and uniformly elliptic diffusion part  have been considered.
In \cite{bjoern-overshoot}, \cite{franke-periodic, franke-periodicerata},  \cite{sandric-spa}, \cite{sandric-rectrans}, \cite{sandric-ergodic} and \cite{sandric-periodic} the authors have derived sufficient conditions for the  transience, recurrence and strong ergodicity of one-dimensional stable-like   processes (see Section \ref{s} for the definition of these processes).  In recent works \cite{sandric-TAMS} and \cite{rene-wang-feller}   Chung-Fuchs type conditions for the  transience and recurrence of L\'evy-type process with bounded coefficients have been derived.
  In  \cite{mas1} the author has obtained conditions for the strong and exponential ergodicity and mixing properties of  strong solutions of L\'evy-driven
stochastic differential equations. Finally, in \cite{kulik}  the  exponential ergodicity  of  a strong solution of pure jump stochastic differential equation (L\'evy-type processes with zero diffusion part) has been studied.

 In this paper, we extend the above mentioned results and obtain general conditions without any further (regularity) assumptions and restrictions on the dimension of the state space and
 coefficients of the operator $\mathcal{L}$.
Also, our conditions are given in terms of the operator $\mathcal{L}$ itself, which
is usually much more accessible and practical.

This paper is organized as follows.
In Section \ref{s2} we give some preliminaries on L\'evy-type processes and in Section \ref{s} we state the main results of this paper. In Section \ref{s3} we discuss conservativeness  of L\'evy-type processes and in Section \ref{s4} we  discuss transience and recurrence of these processes. Finally, in Section \ref{s5}, we discuss  ergodicity and strong, subexponential (polynomial) and exponential ergodicity  of L\'evy-type processes.

\section{Preliminaries on L\'evy-Type Processes}\label{s2}
 Let
$(\Omega,\mathcal{F},\{\mathbb{P}_{x}\}_{x\in\R^{d}},\process{\mathcal{F}},\process{\theta},\process{M})$, denoted by $\process{M}$
in the sequel, be a $d$-dimensional Markov process. A family of linear operators $\process{P}$ on
$B_b(\R^{d})$ (the space of bounded and Borel measurable functions),
defined by $$P_tf(x):= \mathbb{E}_{x}[f(M_t)],\quad x\in\R^{d},\ t\geq0,\
 f\in B_b(\R^{d}),$$ is associated with the process
$\process{M}$. Since $\process{M}$ is a Markov process, the family
$\process{P}$ forms a \emph{semigroup} of linear operators on the
Banach space $(B_b(\R^{d}),\|\cdot\|_\infty)$, that is, $P_s\circ
P_t=P_{s+t}$ and $P_0=I$ for all $s,t\geq0$. Here,
$\|\cdot\|_\infty$ denotes the supremum norm on the space
$B_b(\R^{d})$. Moreover, the semigroup $\process{P}$ is
\emph{contractive}, that is, $\|P_tf\|_{\infty}\leq\|f\|_{\infty}$
for all $t\geq0$ and all $f\in B_b(\R^{d})$, and \emph{positivity
preserving}, that is, $P_tf\geq 0$ for all $t\geq0$ and all $f\in
B_b(\R^{d})$ satisfying $f\geq0$. The \emph{infinitesimal generator}
$(\mathcal{A}^{b},\mathcal{D}_{\mathcal{A}^{b}})$ of the semigroup
$\process{P}$ (or of the process $\process{M}$) is a linear operator
$\mathcal{A}^{b}:\mathcal{D}_{\mathcal{A}^{b}}\longrightarrow B_b(\R^{d})$
defined by
$$\mathcal{A}^{b}f:=
  \lim_{t\longrightarrow0}\frac{P_tf-f}{t},\quad f\in\mathcal{D}_{\mathcal{A}^{b}}:=\left\{f\in B_b(\R^{d}):
\lim_{t\longrightarrow0}\frac{P_t f-f}{t} \ \textrm{exists in}\
\|\cdot\|_\infty\right\}.
$$
A Markov process $\process{M}$ is said to be a \emph{Feller process}
if its corresponding  semigroup $\process{P}$ forms a \emph{Feller
semigroup}. This means that the family $\process{P}$ is a semigroup
of linear operators on the Banach space
$(C_\infty(\R^{d}),\|\cdot\|_{\infty})$  and it is \emph{strongly
continuous}, that is,
  $$\lim_{t\longrightarrow0}\|P_tf-f\|_{\infty}=0,\quad f\in
  C_\infty(\R^{d}).$$ Here, $C_\infty(\R^{d})$ denotes
the space of continuous functions vanishing at infinity.  Note that
every Feller semigroup $\process{P}$  can be uniquely extended to
$B_b(\R^{d})$ (see \cite[Section 3]{rene-conserv}). For notational
simplicity, we denote this extension again by $\process{P}$. Also,
let us remark that every Feller process possesses the strong Markov
property and has (a modification with) c\`adl\`ag sample paths (see  \cite[Theorems 3.4.19 and
3.5.14]{jacobIII}).   Further,
in the case of Feller processes, we call
$(\mathcal{A},\mathcal{D}_{\mathcal{A}}):=(\mathcal{A}^{b},\mathcal{D}_{\mathcal{A}^{b}}\cap
C_\infty(\R^{d}))$ the \emph{Feller generator} for short. Note
that, in this case, $\mathcal{D}_{\mathcal{A}}\subseteq
C_\infty(\R^{\bar{d}})$ and
$\mathcal{A}(\mathcal{D}_{\mathcal{A}})\subseteq
C_\infty(\R^{d})$. If the set of smooth functions
with compact support $C_c^{\infty}(\R^{d})$ is contained in
$\mathcal{D}_{\mathcal{A}}$,
 then, according to \cite[Theorem 3.4]{courrege-symbol},
$\mathcal{A}|_{C_c^{\infty}(\R^{d})}$ is a \emph{pseudo-differential
operator}, that is, it can be written in the form
\begin{align}\label{eq2.1}\mathcal{A}|_{C_c^{\infty}(\R^{d})}f(x) = -\int_{\R^{d}}q(x,\xi)e^{i\langle \xi,x\rangle}
\hat{f}(\xi) d\xi,\end{align}  where $\hat{f}(\xi):=
(2\pi)^{-d} \int_{\R^{d}} e^{-i\langle\xi,x\rangle} f(x) dx$ denotes
the Fourier transform of the function $f(x)$. The function $q :
\R^{d}\times \R^{d}\longrightarrow \CC$ is called  the \emph{symbol}
of the pseudo-differential operator. It is measurable and locally
bounded in $(x,\xi)$ and continuous and negative definite as a
function of $\xi$. Hence, by \cite[Theorem 3.7.7]{jacobI}, the
function $\xi\longmapsto q(x,\xi)$ has for each $x\in\R^{d}$ the
following L\'{e}vy-Khintchine representation \begin{align}\label{eq2.2}q(x,\xi) =a(x)-
i\langle \xi,b(x)\rangle + \frac{1}{2}\langle\xi,c(x)\xi\rangle -
\int_{\R^{d}}\left(e^{i\langle\xi,y\rangle}-1-i\langle\xi,y\rangle1_{B(0,1)}(y)\right)\nu(x,dy),\end{align}
where $a(x)$ is a nonnegative Borel measurable function, $b(x)$ is
an $\R^{d}$-valued Borel measurable function,
$c(x):=(c_{ij}(x))_{1\leq i,j\leq d}$ is a symmetric non-negative
definite $d\times d$ matrix-valued Borel measurable function and $\nu(x,dy)$ is a Borel kernel on $\R^{d}\times
\mathcal{B}(\R^{d})$, called the \emph{L\'evy measure}, satisfying
$$\nu(x,\{0\})=0\quad \textrm{and} \quad \int_{\R^{d}}(1\wedge|y|^{2})\nu(x,dy)<\infty,\quad x\in\R^{d}.$$
The quadruple
$(a(x),b(x),c(x),\nu(x,dy))$ is called the \emph{L\'{e}vy quadruple}
of the pseudo-differential operator
$\mathcal{A}|_{C_c^{\infty}(\R^{d})}$ (or of the symbol $q(x,\xi)$).
Let us remark that the local boundedness of $q(x,\xi)$  implies that for every compact set $K\subseteq\R^{d}$ there exists a finite constant $c_K>0$, such that
\begin{align}\label{eq2.3}\sup_{x\in K}|q(x,\xi)|\leq c_K(1+|\xi|^{2}),\quad \xi\in\R^{d},\end{align} (see \cite[Lemma 3.6.22]{jacobI}).
Moreover, due to \cite[Lemma 2.1 and Remark 2.2]{rene-holder},  \eqref{eq2.3} is equivalent with the local boundedness of the L\'evy quadruple, that is, for every compact set $K\subseteq\R^{d}$ we have \begin{align*}\sup_{x\in K}a(x)+\sup_{x\in K}|b(x)|+\sup_{x\in K}|c(x)|+\sup_{x\in K}\int_{\R^{d}}(1\wedge |y|^{2})\nu(x,dy)<\infty.\end{align*} In addition, according to the same reference, the global boundedness of the L\'evy quadruple is equivalent to $$\|q(\cdot,\xi)\|_\infty\leq c(1+|\xi|^{2}),\quad \xi\in\R^{d},$$ for some
  finite $c>0$.
 Further, note that by combining \eqref{eq2.1} and \eqref{eq2.2},    $\mathcal{A}|_{C_c^{\infty}(\R^{d})}$ has a representation as an integro-differential operator \eqref{eq1.1}.
In the case when the symbol $q(x,\xi)$ does not depend
on the variable $x\in\R^{d}$, $\process{M}$ becomes a \emph{L\'evy
process}, that is, a stochastic process   with stationary and
independent increments and (a modification with) c\`adl\`ag sample paths. Moreover, every L\'evy process is uniquely and completely
characterized through its corresponding symbol (see \cite[Theorems
7.10 and 8.1]{sato-book}). According to this, it is not hard to
check that every L\'evy process satisfies \eqref{eq2.1}  (see \cite[Theorem 31.5]{sato-book}).
Thus, the class of processes we consider in this paper contains a class
 of L\'evy processes. Let us also remark here that, unlike in the case of L\'evy processes, it is not possible to associate a Feller process to every symbol (see  \cite{bjoern-rene-jian} for details). Throughout this paper, the symbol $\process{F}$ denotes a Feller
process satisfying \eqref{eq2.1}. Such a
process is called a \emph{L\'evy-type process}.
If  $\nu(x,dy)=0$ for all $x\in\R^{d}$, according to \cite[Theorem 2.44]{bjoern-rene-jian}, $\process{F}$ becomes an \emph{elliptic diffusion process}.
 For more
on L\'evy-type processes  we refer the readers to the monograph
\cite{bjoern-rene-jian}.

\section{Main Results}\label{s}
In this section, we present the main results of this paper. Before stating the main results, we recall the definitions of transience, recurrence and ergodicity of general Markov processes. Let
$(\Omega,\mathcal{F},\{\mathbb{P}^{x}\}_{x\in\R^{d}},\process{\mathcal{F}},\process{\theta},  \process{M})$, denoted by $\process{M}$
in the sequel, be a Markov process with c\`adl\`ag sample paths and state space
$(\R^{d},\mathcal{B}(\R^{d}))$, where $d\geq1$ and
$\mathcal{B}(\R^{d})$ denotes the Borel $\sigma$-algebra on
$\R^{d}$.
\begin{definition}\label{d1.1}
{\rm The process $\process{M}$  is called
\begin{enumerate}
  \item [(i)] \emph{irreducible} if there exists a $\sigma$-finite measure $\varphi(dy)$ on
$\mathcal{B}(\R^{d})$ such that whenever $\varphi(B)>0$ we have
$\int_0^{\infty}\mathbb{P}^{x}(M_t\in B)dt>0$ for all $x\in\R^{d}$.
  \item [(ii)] \emph{transient} if it is $\varphi$-irreducible
                       and if there exists a countable
                      covering of $\R^{d}$ with  sets
$\{B_j\}_{j\in\N}\subseteq\mathcal{B}(\R^{d})$, such that for each
$j\in\N$ there is a finite constant $c_j\geq0$ such that
$\int_0^{\infty}\mathbb{P}^{x}(M_t\in B_j)dt\leq c_j$ holds for all
$x\in\R^{d}$.
  \item [(iii)] \emph{recurrent} if it is
                      $\varphi$-irreducible and if $\varphi(B)>0$ implies $\int_{0}^{\infty}\mathbb{P}^{x}(M_t\in B)dt=\infty$ for all
                      $x\in\R^{d}$.
\item [(iv)] \emph{Harris recurrent} if it is $\varphi$-irreducible and if $\varphi(B)>0$ implies $\mathbb{P}^{x}(\tau_B<\infty)=1$ for all
                      $x\in\R^{d}$, where $\tau_B:=\inf\{t\geq0:M_t\in
                      B\}.$
\end{enumerate}}
\end{definition}

Let us remark that if $\{M_t\}_{t\geq0}$ is a $\varphi$-irreducible
Markov process, then the irreducibility measure $\varphi(dy)$ can be
maximized. This means that there exists a unique ``maximal" irreducibility
measure $\psi(dy)$ such that for any measure $\bar{\varphi}(dy)$,
$\{M_t\}_{t\geq0}$ is $\bar{\varphi}$-irreducible if, and only if,
$\bar{\varphi}\ll\psi$ (see \cite[Theorem 2.1]{tweedie-mproc}).
According to this, from now on, when we refer to irreducibility
measure we actually refer to the maximal irreducibility measure. In
the sequel, we consider  only the so-called \emph{open-set irreducible}
Markov processes, that is,
Markov processes whose maximal irreducibility measure is fully supported.
An example of  such a measure is  Lebesgue measure, which we denote by $\lambda(dy)$. Clearly,  a Markov process $\process{M}$ will be
$\lambda$-irreducible if $\mathbb{P}^{x}(M_t\in B)>0$ for all
 $x\in\R^{d}$ and $t>0$ whenever $\lambda(B)>0.$ In particular, the
process $\process{M}$ will be $\lambda$-irreducible if the
transition kernel $\mathbb{P}^{x}(M_t\in dy)$, $x\in\R^{d}$, $t>0$, possesses a strictly positive transition density
function with respect to $\lambda(dy)$.
 Let us remark here that irreducibility  of L\'evy-type processes is a very well-studied topic in the literature. In particular, we refer the readers to \cite{sheu} and \cite{stramer-tweedie-sinica} for the elliptic diffusion case, to \cite{vasili-stablelike} and \cite{vasili-book} for the case of  stable-like processes (see Example \ref{e1.6} for the definition of these processes),
to \cite{victoria3}, \cite{victoria2}, \cite[Remark 3.3]{turbul} and \cite[Theorem 2.6]{sandric-TAMS}   for the case of L\'evy-type processes with bounded coefficients and to \cite{bass},  \cite{ish}, \cite{victoria1}, \cite{kulik}, \cite{mas1, mas2}  and \cite{pic1, pic2} for the case of a class of L\'evy-type processes obtained as a solution of certain jump-type stochastic differential equations.
Further, it is well known that every $\psi$-irreducible Markov
process is either  transient or recurrent (see \cite[Theorem
2.3]{tweedie-mproc}). Also, clearly, every Harris recurrent Markov
process is recurrent, but, in general, these two properties are not
equivalent. They differ on the set of the irreducibility measure
zero (see \cite[Theorem 2.5]{tweedie-mproc}). However,
 for an open-set irreducible L\'evy-type process  these two
 properties are actually equivalent (see  Proposition
  \ref{p4.1}).

Now, we recall  notions of ergodicity  of Markov processes. A probability measure $\pi(dx)$ on
$\mathcal{B}(\R^{d})$ is called \emph{invariant} for $\process{M}$ if
$$\int_{\R^{d}}\mathbb{P}^{x}(M_t\in B)\pi(dx)=\pi(B),\quad t>0,\  B\in\mathcal{B}(\R^{d}).$$  A set $B\in\mathcal{F}$ is said to be
\emph{shift-invariant} if $\theta_t^{-1}B=B$ for all $t\geq0$. The
\emph{shift-invariant} $\sigma$-algebra $\mathcal{I}$ is a
collection of all such shift-invariant sets.

\begin{definition}\label{d1.2}
{\rm The process
$\process{M}$ is called
\begin{itemize}
  \item [(i)] \emph{ergodic} if it possesses an invariant probability measure $\pi(dx)$ and if  $\mathcal{I}$ is
trivial with respect to $\mathbb{P}^{\pi}(d\omega)$, that is,
$\mathbb{P}^{\pi}(B)=0$ or $1$ for every $B\in\mathcal{I}$. Here,
for a probability measure $\mu(dx)$ on $\mathcal{B}(\R^{d})$,
$\mathbb{P}^{\mu}(d\omega)$ is defined as
$\mathbb{P}^{\mu}(d\omega):=\int_{\R^{d}}\mathbb{P}^{x}(d\omega)\mu(dx).$
  \item [(ii)]\emph{strongly ergodic} if it possesses an invariant probability measure $\pi(dx)$ and if
$$\lim_{t\longrightarrow\infty}\|\mathbb{P}^{x}(M_t\in\cdot)-\pi(\cdot)\|_{TV}=0,\quad  x\in \R^{d},$$ where $\|\cdot\|_{TV}$ denotes the
total variation norm on the space of  signed measures on $\mathcal{B}(\R^{d})$.
  \item [(iii)] \emph{polynomially ergodic}  if it possesses an invariant probability measure $\pi(dx)$ and if
$$\|\mathbb{P}^{x}(M_t\in\cdot)-\pi(\cdot)\|_{TV}\leq k(x)t^{-\kappa},\quad x\in \R^{d},\ t\geq0,$$ for some   $k:\R^{d}\longrightarrow[0,\infty)$ and $\kappa>0$.
\item [(iv)] \emph{exponentially ergodic}  if it possesses an invariant probability measure $\pi(dx)$ and if
$$\|\mathbb{P}^{x}(M_t\in\cdot)-\pi(\cdot)\|_{TV}\leq k(x)e^{-\kappa t},\quad x\in \R^{d},\ t\geq0,$$ for some   $k:\R^{d}\longrightarrow[0,\infty)$ and $\kappa>0$.
\end{itemize}}
\end{definition}
Clearly,  exponential ergodicity implies polynomial ergodicity, which implies  strong ergodicity and strong ergodicity implies ergodicity (for the latter see \cite[Proposition 2.5]{bhat}). On the other hand, ergodicity does not necessarily imply strong ergodicity, strong ergodicity does not  imply polynomial ergodicity which in general does not imply exponential ergodicity (see  \cite{subgeometric}, \cite{meyn-tweedie-book} and \cite{turbul}). However,
 for an open-set irreducible  L\'evy-type process which has an irreducible \emph{skeleton chain}, ergodicity and strong ergodicity actually coincide (see  Section
  \ref{s5}).
Recall, a Markov process $\process{M}$ has an irreducible skeleton chain if  there are $t_0>0$ and $\sigma$-finite measure $\varphi(dy)$ on
$\mathcal{B}(\R^{d})$, such that the Markov chain $\{M_{nt_0}\}_{n\geq0}$ is $\varphi$-irreducible, that is,
whenever $\varphi(B)>0$ we have
$\sum_{n=0}^{\infty}\mathbb{P}^{x}(M_{nt_0}\in B)>0$ for all $x\in\R^{d}$.

 Now, we are in a position to state the main results of this paper, the proofs of which are given in Sections \ref{s4} and \ref{s5}. First, we introduce some auxiliary notation we need in the sequel.
For $\alpha\geq0$, $r_0>1$, $0<\varepsilon\leq1-r_0^{-\alpha}$, $r\geq r_0$ and $x\in\R^{d}$, $|x|\geq r_0$, define \begin{align*}W_\alpha(r):=1-r^{-\alpha},\quad V_\alpha(r):=\left\{
                                                                 \begin{array}{ll}
                                                                   \ln r, & \alpha=0 \\
                                                                   r^{\alpha}, & \alpha>0,
                                                                 \end{array}
                                                               \right.  \end{align*}

\begin{align*}
D^{W}_\alpha(x)&:=\frac{\alpha}{2}\left((|x|+1)^{-2-\alpha}-(2+\alpha)\frac{(|x|+1)^{2}}{(|x|-1)^{4+\alpha}}\right)\int_{B(0,1)}|y|^{2}\nu(x,dy), \nonumber\\
D^{V}_\alpha(x)&:=\left\{
                 \begin{array}{ll}
                \displaystyle{ \frac{1}{2} (|x|-1)^{-2}\int_{B(0,1)}|y|^{2}\nu(x,dy)} , & \alpha=0 \nonumber\\
                   \displaystyle{\frac{\alpha}{2}(|x|-1)^{-2+\alpha}\int_{B(0,1)}|y|^{2}\nu(x,dy)}, & 0<\alpha\leq2  \nonumber\\
                  \displaystyle{\frac{\alpha}{2} \left((|x|-1)^{-2+\alpha}-(2-\alpha)\frac{(|x|+1)^{2}}{(|x|-1)^{4-\alpha}}\right)\int_{B(0,1)}|y|^{2}\nu(x,dy)}, & 2<\alpha\leq4 \nonumber\\
     \displaystyle{ \frac{\alpha}{2}\left((|x|-1)^{-2+\alpha}-(2-\alpha)(|x|+1)^{\alpha-2}\right)\int_{B(0,1)}|y|^{2}\nu(x,dy)}, & \alpha>4,
 \end{array}
               \right.
\end{align*}
\begin{align}\label{eq1.4}
 E^{W}_\alpha(x)&:=\int_{B^{c}(0,1)}\left(W_\alpha(|y+x|)1_{\{|y+x|>r_0\}}(y)+(W_\alpha(r_0)-\varepsilon)1_{\{|y+x|\leq r_0\}}(y)-W_\alpha(|x|)\right)\nu(x,dy),\nonumber\\
E^{V}_\alpha(x)&:=\int_{B^{c}(0,1)}(V_\alpha(|y+x|\vee r_0)-V_\alpha(|x|))\nu(x,dy), \nonumber\\
T_\alpha(x)&:=-a(x)W_\alpha(|x|)+\alpha |x|^{-\alpha}(A(x)-(1+\alpha/2)C(x)+B(x))+D_\alpha^{W}(x)+E_\alpha^{W},\nonumber\\[0.4em]
R_\alpha(x)&:=\left\{
                            \begin{array}{ll}
                              \displaystyle{-a(x)V_0(|x|)+A(x)-C(x)+B(x)+D_0^{V}(x)+E_0^{V}(x)}, & \alpha=0 \\ [0.4em]
                             \displaystyle{ -a(x)V_\alpha(|x|)+\alpha |x|^{\alpha}(A(x)-(1-\alpha/2)C(x)+B(x))+D_\alpha^{V}(x)+E_\alpha^{V}}, & \alpha>0,
                            \end{array}
                          \right.
\end{align}
where $a(x)$, $A(x)$, $B(x)$ and $C(x)$ are defined in \eqref{eq1.1} and \eqref{eq1.2},  $B(x,r)$ denotes the open ball around $x\in\R^{d}$ of radius $r>0$ and $a\wedge b$ and $a\vee b$ denote the minimum and maximum of $a,b\in\R$, respectively.
It is straightforward to see (by employing Taylor's formula) that for any $C^{2}$-extensions $\bar{W}_\alpha, \bar{V}_\alpha:\R^{d}\longrightarrow[0,\infty)$  of the functions $x\longmapsto W_\alpha(|x|)$ and $x\longmapsto V_\alpha(|x|)$, $x\in\R^{d}$, $|x|\geq r_0$, such that  the functions $|x|\longrightarrow \bar{W}_\alpha(x)$ and $|x|\longrightarrow \bar{V}_\alpha(x)$ are nondecreasing and $\bar{W}_\alpha(0)=1-r_0^{-\alpha}-\varepsilon$, we have that $\mathcal{L}\bar{W}_\alpha(x)\geq T_\alpha(x)$ and $\mathcal{L}\bar{V}_\alpha(x)\leq R_\alpha(x)$, $x\in\R^{d}$, $|x|\geq r_0$.
Also, observe that
\begin{itemize}
  \item [(i)] if $\alpha\leq1$, then for all $x\in\R^{d}$, $|x|\geq r_0$,
  $$E^{W}_{\alpha}(x)\geq-\frac{\alpha}{r_0|x|^{\alpha}}\int_{\{|y|\geq1,\, |x|>|y+x|>r_0\}}|y|\nu(x,dy)+(-1+|x|^{-\alpha})\nu(x,\{|y|\geq1,\, |y+x|\leq r_0\}).$$
  \item [(ii)]  for all $x\in\R^{d}$, $|x|\geq r_0$,
$$E^{V}_\alpha(x)\leq\left\{
                 \begin{array}{ll}
                \int_{B^{c}(0,1)}\ln\left(1+\frac{|y|}{|x|}\right)\nu(x,dy) , & \alpha=0\\[0.4em]
                  \int_{B^{c}(0,1)}\left((|x|+|y|)^{\alpha}-|x|^{\alpha}\right)\nu(x,dy), & \alpha>0.
 \end{array}\right.$$
 In particular,
 if $\alpha\leq1$, then for all $x\in\R^{d}$, $|x|\geq r_0$,
 $$E^{V}_\alpha(x)\leq\left\{
                 \begin{array}{ll}
                |x|^{-1}\int_{B^{c}(0,1)}|y|\nu(x,dy) , & \alpha=0\\[0.4em]
                  \alpha|x|^{-1+\alpha}\int_{B^{c}(0,1)}|y|\nu(x,dy), & 0<\alpha\leq1.
 \end{array}\right.$$

\end{itemize}

\begin{theorem}\label{tm1.3}Let $\process{F}$ be a $d$-dimensional open-set irreducible L\'evy-type process generated by an operator of the form \eqref{eq1.1} with coefficients $(a(x),b(x),c(x),\nu(x,dy))$.
\begin{itemize}
 \item [(i)] The process $\process{F}$ is transient if there exist $\alpha>0$, $x_0>r_0>1$ and $0<\varepsilon\leq1-r_0^{-\alpha}$, such that $T_\alpha(x)\geq0$ for all $x\in\R^{d}$, $|x|\geq x_0$.
  \item [(ii)]The process $\process{F}$ is recurrent if there exist $\alpha\geq0$ and $x_0>r_0>1$, such that
\begin{align}\label{eq1.5}z\longmapsto\int_{\{|y|\geq1,\, |y+z|\geq x_0\}}V_\alpha(|y+z|)\nu(z,dy)\quad\textrm{is locally bounded}
\end{align}  and $R_\alpha(x)\leq0$ for all $x\in\R^{d}$, $|x|\geq x_0$.
  \item [(iii)]The process $\process{F}$ is (strongly) ergodic if it has an irreducible skeleton chain and if there exist $\alpha\geq0$, $\beta>0$ and $x_0>r_0>1$, such that \eqref{eq1.5} holds true and $R_\alpha(x)\leq-\beta$ for all $x\in\R^{d}$, $|x|\geq x_0$.
  \item [(iv)] The process $\process{F}$ is polynomially ergodic if  it has an irreducible skeleton chain and if there exist $\alpha\geq0$, $0<\beta<1$, $\gamma>0$ and $x_0>r_0>1$, such that \eqref{eq1.5} holds true and $R_\alpha(x)\leq-\gamma V^{\beta}_\alpha(|x|)$ for all $x\in\R^{d}$, $|x|\geq x_0$. In this case, the corresponding polynomial rate of convergence is $t^{\beta/(1-\beta)}$, for any $0<\lambda<\gamma(1-\beta)$ and $t_0>0$ there exists $k>0$ such that
   \begin{align}\label{eq1.9}&\|\mathbb{P}^{x}(X_t\in\cdot)-\pi(\cdot)\|_{TV}\nonumber\\&\leq k(1-\beta)\left(t_0^{1/(1-\beta)}+\frac{\lambda^{\beta/(\beta-1)}}{\gamma(1-\beta)-\lambda} \bar{V}_\alpha(x)+\frac{t_0\lambda^{\beta/(\beta-1)}}{\gamma(1-\beta)-\lambda} \sup_{B(0,x_0)}|\mathcal{L}\bar{V}_\alpha(x)|\right)t^{-\beta/(1-\beta)}\end{align} for all $ x\in \R^{d}$ and  $t\geq0$,
       and \begin{align}\label{eq1.6}\int_{\R^{d}}\bar{V}^{\beta}_\alpha(x)\pi(dx)\leq\sup_{x\in B(0,x_0)}|\mathcal{L}\bar{V}_\alpha(x)|/\gamma+V^{\beta}_\alpha(x_0),\end{align}  where $\bar{V}_\alpha(x)$ is  any $C^{2}$-extension of the function $x\longmapsto V_\alpha(|x|)$, $x\in\R^{d}$, $|x|\geq r_0$, such that the function $|x|\longrightarrow \bar{V}_\alpha(x)$ is nondecreasing and $\bar{V}_\alpha(0)>0$.
\item [(v)] The process $\process{F}$ is exponentially ergodic if it has an irreducible skeleton chain and if there exist $\alpha\geq0$, $\beta>0$ and $x_0>r_0>1$, such that \eqref{eq1.5} holds true and $R_\alpha(x)\leq-\beta V_\alpha(|x|)$ for all $x\in\R^{d}$, $|x|\geq x_0$. In this case, for any $0<\lambda<\beta$, $t_0>0$ and $\kappa>0$ there exists $k(\kappa)>0$ such that
 \begin{align}\label{eq1.7}\|\mathbb{P}^{x}(X_t\in\cdot)-\pi(\cdot)\|_{TV}\leq \left(1+ \frac{e^{\lambda t_0}}{\beta-\lambda}\bar{V}_\alpha(x)+\frac{t_0e^{\lambda t_0}}{\beta-\lambda}\sup_{B(0,x_0)}|\mathcal{L}\bar{V}_\alpha(x)|+\frac{e^{\lambda t_0}-1}{\lambda}\right)e^{k(\kappa)-\kappa t}\end{align} for all $ x\in \R^{d}$ and  $t\geq0$,
and \begin{align}\label{eq1.8}\int_{\R^{d}}\bar{V}_\alpha(x)\pi(dx)\leq\sup_{x\in B(0,x_0)}|\mathcal{L}\bar{V}_\alpha(x)|/\beta+V_\alpha(r_0).\end{align} Here, $\bar{V}_\alpha(x)$ is again any $C^{2}$-extension of the function $x\longmapsto V_\alpha(|x|)$, $x\in\R^{d}$, $|x|\geq r_0$, such that  the function $|x|\longrightarrow \bar{V}_\alpha(x)$ is nondecreasing and $\bar{V}_\alpha(0)\geq0$.
\end{itemize}
\end{theorem}
Let us remark here that  the condition in \eqref{eq1.5}  can be relaxed by replacing $V_0(r)$ with
$$V_{-n+1}(r):=\ln\ln\cdots\ln r,\quad r\geq r_0>e^{n-1},\ n\in\N.$$
Clearly, in that case, Theorem \ref{tm1.3}  still holds, but with some minor technical modifications.

As a direct consequence of Theorem \ref{tm1.3} we can also discuss mixing properties of L\'evy-type processes (see also  \cite{mas1}). First, recall that $\alpha$\emph{-mixing} (or strong mixing) and $\beta$\emph{-mixing} (or complete regularity, or
the Kolmogorov) \emph{coefficients} of a  Markov process $\process{M}$ with initial distribution $\mu(dx)$ are defined as follows
\begin{align*}\alpha^{\mu}(t)&:=\sup_{s\geq0}\sup_{A\in\mathcal{F}_s,\, B\in\sigma\{M_u:\,u\geq s+t\}}|\mathbb{P}^{\mu}(A\cap B)-\mathbb{P}^{\mu}(A)\mathbb{P}^{\mu}(B)|\\\beta^{\mu}(t)&:=\sup_{s\geq0}\mathbb{E}^{\mu}\left[\sup_{B\in\sigma\{M_u:\,u\geq s+t\}}|\mathbb{P}^{\mu}(B|\mathcal{F}_s)-\mathbb{P}^{\mu}(B)|\right].\end{align*}
It is well known that $\alpha^{\mu}_t\leq\beta^{\mu}_t$ for every $t\geq0$ and every initial distribution $\mu(dy)$ of $\process{M}$  (see \cite{brad}).
Further, if $\pi(dx)$ is an invariant  distribution of $\process{M}$, then, by using the Markov property of $\process{M}$ and stationarity of $\pi(dx)$, we have the following
\begin{align*}\alpha^{\pi}(t)&=\sup_{A\in\mathcal{F}_0,\, B\in\sigma\{M_u:\,u\geq t\}}|\mathbb{P}^{\pi}(A\cap B)-\mathbb{P}^{\pi}(A)\mathbb{P}^{\pi}(B)| \\ \beta^{\pi}(t)&=\int_{\R^{d}}\|\mathbb{P}^{x}(M_t\in \cdot)-\pi(\cdot)\|_{TV}\pi(dx)\end{align*} (see \cite{brad} and \cite{dav}).
A Markov process $\process{M}$ with initial distribution $\mu(dx)$ is called $\alpha^{\mu}$\emph{-mixing} (respectively, $\beta^{\mu}$\emph{-mixing}) if $\lim_{t\longrightarrow\infty}\alpha^{\mu}(t)=0$ (respectively, $\lim_{t\longrightarrow\infty}\beta^{\mu}(t)=0$).
\begin{corollary}Let $\process{F}$ be a L\'evy-type process satisfying the assumptions from Theorem \ref{tm1.3} (iii). Then, $\process{F}$ is $\beta^{\pi}$-mixing. Furthermore, if $\process{F}$ satisfies the assumptions from Theorem \ref{tm1.3} (v), then
$\lim_{t\longrightarrow\infty}e^{\kappa t}\beta^{\pi}(t)=0$ for every $\kappa>0.$ Here,  $\pi(dx)$ denotes the unique invariant distribution of $\process{X}$.
\end{corollary}

Let us now give some applications of the results presented above.
\begin{example}[Elliptic diffusions]\label{e1.5}{\rm
Assume that the coefficients $b(x)=(b_i(x))_{1\leq i\leq d}$ and $c(x)=(c_{ij}(x))_{1\leq i,j\leq d}$ satisfy the following:
 \begin{itemize}
  \item [(i)] $b(x)$ is continuous;
   \item [(ii)] $c(x)$ is   symmetric and Lipschitz continuous;
   \item [(ii)] for some constant $\kappa\geq1$ and all $i,j=1,\ldots,d$ and $x\in\R^{d}$, $$|b_i(x)|+|c_{ij}(x)|^{1/2}\leq \kappa(1+|x|)\quad\textrm{and}\quad \kappa^{-1}\sum_{i=1}^{d}\xi^{2}_i\leq\sum_{i,j=1}^{d}\xi_i\xi_jc_{ij}(x)\leq \kappa\sum_{i=1}^{d}\xi^{2}_i,\quad \xi\in\R^{d}.$$
 \end{itemize}
  Then, according to \cite[Theorem V.24.1]{rogers} and \cite[Theorem 2.3]{stramer-tweedie-sinica}, the operator $\mathcal{L}$ (with coefficients $b(x)$ and $c(x)$) generates a unique open-set irreducible  elliptic diffusion process which has an irreducible skeleton chain. Thus, we are in position to  apply Theorem \ref{tm1.3}. Specially, as a simple consequence we can deduce the
well-known transience and recurrence dichotomy of a standard Brownian motion, that is, a standard Brownian motion is transient if, and only if, $d>2$.}
\end{example}

\begin{example}[Stable-like processes]\label{e1.6}
\rm{
Let $\alpha:\R^{d}\longrightarrow(0,2)$, $\beta:=(\beta_i)_{1\leq i\leq d}:\R^{d}\longrightarrow\R^{d}$ and $\gamma:\R^{d}\longrightarrow(0,\infty)$   be such that:
\begin{itemize}
  \item [(i)] $\alpha,\beta_i,\gamma\in C^{1}_b(\R^{d})$, $i=1,\ldots,d$, where $C_b^{k}(\R^{d})$, $k\geq 0$, denotes the space of $k$ times differentiable functions such that all derivatives up to order $k$ are bounded;
  \item [(ii)]$0<\inf_{x\in\R^{d}}\alpha(x)\leq\sup_{x\in\R^{d}}\alpha(x)<2$ and $\inf_{x\in\R^{d}}\gamma(x)>0$.
\end{itemize}
  Then, under this assumptions, in
\cite{bass-stablelike}, \cite[Theorem 5.1]{vasili-stablelike} and
                                                         \cite[Theorem 3.3.]{rene-wang-feller}
                                                         it has been shown
                                                         that there
                                                         exists a
                                                         unique open-set irreducible  L\'evy-type
                                                         process which has an irreducible skeleton chain,
                                                         called a
                                                         \emph{stable-like
                                                         process},
                                                         determined
by coefficients of the form  $(0,\beta(x),0,\gamma(x)|y|^{-d-\alpha(x)}dy)$.  Note that when
$\alpha(x)$, $\beta(x)$ and $\gamma(x)$ are constant functions, then we deal
with a symmetric stable L\'evy process with drift. Now, by a straightforward application of Theorem \ref{tm1.3}, it is easy to see that
\begin{itemize}
\item [(i)]if for some $\alpha\leq1$ and $r_0>1$,   $$\liminf_{|x|\longrightarrow\infty}\left(\sum_{i=1}^{d}x_i\beta_i(x)-\frac{\alpha(\alpha+1)S_d\gamma(x)}{2(2-\alpha(x))}- V_dr_0^{d-1}\gamma(x)|x|^{2-\alpha(x)-d}(\alpha|x|+r_0|x|^{\alpha})\right)>0,$$ then the underlying stable-like process is transient;
  \item [(ii)] if $\liminf_{|x|\longrightarrow\infty}\alpha(x)>1$ and $$\limsup_{|x|\longrightarrow\infty}\left(\sum_{i=1}^{d}x_i\beta_i(x)+\frac{S_d\gamma(x)}{2(2-\alpha(x))}+\frac{S_d\gamma(x)}{\alpha(x)-1}|x|\right)<0,$$ then the underlying stable-like process is recurrent (here we used $V_0(r)=\ln r$);
      \item [(iii)]if for some $1\leq\alpha<2$ and $\beta>0$, $\liminf_{|x|\longrightarrow\infty}\alpha(x)>\alpha$ and $$\limsup_{|x|\longrightarrow\infty}\left(\alpha\sum_{i=1}^{d}x_i\beta_i(x)+\frac{\alpha S_d\gamma(x)}{2(2-\alpha(x))}+\frac{\alpha S_d\gamma(x)}{\alpha(x)-1}|x|+\frac{\alpha S_d\gamma(x)}{\alpha(x)-\alpha}|x|^{2-\alpha}+\beta|x|^{2-\alpha}\right)<0,$$ then the underlying stable-like process is (strongly) ergodic (here we used $V_\alpha(r)= r^{\alpha}$);
\item [(iv)]if for some $1\leq\alpha<2$, $0<\beta\leq(\alpha-1)/\alpha$ and $\gamma>0$, $\liminf_{|x|\longrightarrow\infty}\alpha(x)>\alpha$ and $$\limsup_{|x|\longrightarrow\infty}\left(\alpha\sum_{i=1}^{d}x_i\beta_i(x)+\frac{\alpha S_d\gamma(x)}{2(2-\alpha(x))}+\frac{\alpha S_d\gamma(x)}{\alpha(x)-1}|x|+\frac{\alpha S_d\gamma(x)}{\alpha(x)-\alpha}|x|^{2-\alpha}+\gamma|x|^{2-\alpha+\alpha\beta}\right)<0,$$ then the underlying stable-like process is polynomially ergodic with the rate of convergence $t^{\beta/(1-\beta)}$ (here we used $V_\alpha(r)= r^{\alpha}$).
\end{itemize}
Here, $S_d$ and $V_d$ denote the surface and volume of a $d$-dimensional unit ball, respectively. Also, note that, because of the boundedness of the coefficients, it is not clear that a stable-like process can be exponentially ergodic.}
\end{example}

\begin{example}[Ornstein-Uhlenbeck-type processes]\label{e1.7}
\rm{ Let $q:=(q_{ij})_{1\leq i,j\leq d}$ be a $d\times d$ real matrix  whose
 eigenvalues all have strictly  positive real parts and let $\process{L}$ be an $\R^{d}$-valued    L\'evy process determined by coefficients   $(0,b,c,\nu(dy))$. Furthermore, let $F_0$ be an $\R^{d}$-valued random variable   independent of $\process{L}$.  The \emph{Ornstein-Uhlenbeck-type processes} is a strong Markov process defined by
$$F_t:=e^{-tq}F_0+\int_0^{t}e^{-(t-s)q}dL_s,\quad t\geq0,$$ (see \cite{sato} for details).
Under  certain regularity conditions of the matrix $q$ and coefficients $(0,b,c,\nu(dy))$, in \cite[Theorem 3.1]{sato} it has been shown that $\process{F}$ is an open-set irreducible  L\'evy-type process which has an irreducible skeleton chain, determined by coefficients of the form $(0,b-qx,c,\nu(dy))$. Now, assume that there exists a constant $\kappa>0$, such that $$\sum_{i,j=1}^{d}x_ix_jq_{ij}(x)\geq \kappa\sum_{i=1}^{d}x^{2}_i,\quad x\in\R^{d}.$$ Then,
$\process{F}$ is (strongly) ergodic if, and only if,
$\nu(dy)$ satisfies \eqref{eq1.5} with $V_0(r)=\ln(r)$. The necessity has been proved in \cite[Theorem 4.2]{sato}, while the sufficiency  easily follows from Theorem \ref{tm1.3}. Let us also remark here that, under the condition in \eqref{eq1.5} (with $V_0(r)=\ln(r)$), in \cite[Theorems 4.1]{sato} the authors have explicitly determined the corresponding invariant measure. Finally,  if $\nu(dy)$ satisfies \eqref{eq1.5} with $V_\alpha(r)=r^{\alpha}$, for some $\alpha>0$, then again by a straightforward application of Theorem \ref{tm1.3}  it is easy to see that $\process{F}$ is exponentially ergodic.
}\end{example}

\begin{example}[L\'evy-driven SDEs]\label{e1.8}
\rm{Let $\process{L}$ be an $n$-dimensional L\'evy process and let $\Phi:\R^{d}\longrightarrow\R^{d\times n}$ be bounded and locally Lipschitz continuous. Then, the SDE $$dF_t=\Phi(F_{t-})dL_t,\quad F_0=x\in\R^{d},$$ admits a unique strong solution which is a L\'evy-type process (see \cite[Theorem 3.8]{bjoern-rene-jian}). In particular, if
\begin{enumerate}
  \item [(i)] $L_t=(l_t,t)$, $t\geq0$, where $\process{l}$ is a $d$-dimensional L\'evy process determined by coefficients (L\'evy triplet) $(0,b,c,\nu(dy))$  such that the L\'evy measure $\nu(dy)$ is symmetric;
  \item [(ii)] $\Phi(x)=(\phi(x)I,\psi(x))$, $x\in\R^{d}$, where $\phi,\psi:\R^{d}\longrightarrow\R$  are bounded, locally Lipschitz continuous and $|\phi(x)|>0$ for all $x\in\R^{d}$,
\end{enumerate}
then  $\process{F}$
is a $d$-dimensional L\'evy-type process  determined
by coefficients of the form  $$(0,\psi(x)+\phi(x)b,|\phi(x)|^{2}c,\nu(dy/|\phi(x)|)).$$  The open-set irreducibility and existence of an  irreducible skeleton chain of $\process{F}$ (in terms of $\Phi(x)$)  have been discussed in \cite{mas1}.
Thus, we are again in position to apply Theorem \ref{tm1.3}.
}\end{example}

The two main ingredients in proving Theorem \ref{tm1.3} are: (i) characterizations of   transience, recurrence, ergodicity and strong, polynomial and exponential erodicity in terms of the first hitting time $\tau_{B(0,x_0)}$ (see Propositions \ref{p4.1}, \cite[Theorem 4.4]{meyn-tweedie-II}, \cite[Theorem 1]{subgeometric} and \cite[Theorem 6.2]{down}) and (ii) noticing that for any $f\in C^{2}(\R^{d})$ satisfying the condition in \eqref{eq1.5} the process
$$\left\{f(F_{t\wedge\tau_{B(0,r_0)}})-\int_0^{t\wedge\tau_{B(0,r_0)}}\mathcal{L}f(F_s)ds\right\}_{t\geq0}$$ is a $\mathbb{P}^{x}$-local martingale for all $x\in\R^{d}$ (the operator $\mathcal{L}$ will be an extension of the infinitesimal generator of $\process{F}$ on this class of functions). Then, by an appropriate choice of the function $f(x)$ (that is, $f(x)=\bar{V}_\alpha(x)$ or $\bar{W}_\alpha(x)$, where $\bar{V}_\alpha(x)$ and $\bar{W}_\alpha(x)$ are adequate $C^{2}$-extensions of the functions $x\longmapsto V_\alpha(|x|)$ and $x\longmapsto W_\alpha(|x|)$, respectively) and analysis of this process, we are in position to derive the desired conditions presented in Theorem \ref{tm1.3}.

\section{Conservativeness}\label{s3}
In this section, we discuss  conservativeness  of L\'evy-type processes.
Let  $\process{M}$
 be a $d$-dimensional  Markov process
and define
$$T_c:=\inf\{t\geq0: M_t\notin\R^{d}\}\quad\textrm{and}\quad T_e:=\lim_{R\longrightarrow\infty}\inf\{t\geq0:
M_t\in B^{c}(0,R)\}.$$  The process $\process{M}$ is called
\emph{conservative}  if $\mathbb{P}^{x}(T_c=\infty)=1$ for all
$x\in\R^{d}$ and \emph{nonexplosive} if
$\mathbb{P}^{x}(T_e=\infty)=1$ for all $x\in\R^{d}$.
Observe that, due to the fact that $\process{M}$ has c\`adl\`ag sample paths,  these two notions actually coincide and
 they are equivalent with the fact that $\mathbb{P}^{x}(M_t\in\R^{d})=1$ for all $x\in\R^{d}$ and $t\geq0$ (see \cite{rene-conserv}). Also, note that
$T_c$ represents the moment the process ceases to be
finitely valued. Usually, once $T_c$ has been reached we \emph{kill} the process. This can be accomplished by a one-point compactification of the state space $\R^{d}$, say $\R^{d}_\infty,$ and by defining
$$\mathbb{P}^{x}(M_t\in B):=\left\{
  \begin{array}{ll}
    1, & x=\infty\ \ \textrm{and}\ \ B=\{\infty\} \\
    0, & x=\infty\ \ \textrm{and}\ \ B=\R^{d}.
  \end{array}
\right.$$

\begin{proposition}\label{p3.1}Let $\process{M}$ be a $d$-dimensional Markov process. If the function $x\longmapsto \mathbb{P}^{x}(T_c<\infty)$ is lower semicontinuous, that is, $\liminf_{y\longrightarrow x}\mathbb{P}^{y}(T_c<\infty)\geq\mathbb{P}^{x}(T_c<\infty)$ for all $x\in\R^{d},$ and $\int_0^{\infty}\mathbb{P}^{x}(M_t\in O)dt>0$ for all $x\in\R^{d}$ and open sets $O\subseteq\R^{d}$, then $\mathbb{P}^{x}(T_c=\infty)=1$ for some $x\in\R^{d}$ if, and only if, $\mathbb{P}^{x}(T_c=\infty)=1$ for all $x\in\R^{d}$.

Similarly, if the function $x\longmapsto \mathbb{P}^{x}(T_c<\infty)$ is upper semicontinuous, that is, $\limsup_{y\longrightarrow x}\mathbb{P}^{y}(T_c<\infty)\leq\mathbb{P}^{x}(T_c<\infty)$ for all $x\in\R^{d},$ and $\int_0^{\infty}\mathbb{P}^{x}(M_t\in O)dt>0$ for all $x\in\R^{d}$ and open sets $O\subseteq\R^{d}$, then $\mathbb{P}^{x}(T_c=\infty)=0$ for some $x\in\R^{d}$ if, and only if, $\mathbb{P}^{x}(T_c=\infty)=0$ for all $x\in\R^{d}$.
\end{proposition}
\begin{proof}
Assume that $\mathbb{P}^{x_0}(T_c<\infty)>0$ for some $x_0\in\R^{d}$. Then, because of the lower semicontinuity of  $x\longmapsto \mathbb{P}^{x}(T_c<\infty)$, there exists an open set $O_0\subseteq\R^{d}$ around $x_0$ such that $\inf_{x\in O_0}\mathbb{P}^{x}(T_c<\infty)>0$. Next, let $x\in\R^{d}$ and $t>0$ be arbitrary. Then, by the Markov property, we have \begin{align*}\mathbb{P}^{x}(T_c<\infty)&=\mathbb{P}^{x}\left(\inf\{s\geq t:M_s\notin\R^{d}\}<\infty\right)
\\&=\int_{\R^{d}}\mathbb{P}^{x}(M_t\in dy)\mathbb{P}^{y}(T_c<\infty)+\mathbb{P}^{x}(M_t=\infty)\\&\geq\int_{O_0}\mathbb{P}^{x}(M_t\in dy)\mathbb{P}^{y}(T_c<\infty)\\&\geq\inf_{y\in O_0}\mathbb{P}^{y}(T_c<\infty)\mathbb{P}^{x}(M_t\in O_0).\end{align*}
Now, by assumption,  for given $x\in\R^{d}$ and $O_0\subseteq\R^{d}$, there exists $T_0>0$ such that $\int_0^{T_0}\mathbb{P}^{x}(M_t\in O_0)dt>0$. Hence,
$\mathbb{P}^{x}(T_c<\infty)>0$ for all $x\in\R^{d}$, which leads to a contradiction.

 To prove  the second assertion, note first that the
 upper semicontinuity of the function $x\longmapsto\mathbb{P}^{x}(T_c<\infty)$ is equivalent with the lower semicontinuity of the function $x\longmapsto\mathbb{P}^{x}(T_c=\infty).$ Now, the claim follows by  completely the same reasoning as above.
\end{proof}

Recall that a
semigroup $\process{P}$ on $(B_b(\R^{d}),\|\cdot\|_\infty)$ is
called a \emph{$C_b$-Feller semigroup} if $P_t(C_b(\R^{d}))\subseteq
C_b(\R^{d})$ for all $t\geq0$ and it is called a \emph{strong Feller
semigroup} if $P_t(B_b(\R^{d}))\subseteq C_b(\R^{d})$ for all
$t\geq0$. For sufficient conditions for a Feller semigroup
to be a $C_b$-Feller semigroup or a strong Feller semigroup see
\cite{rene-conserv} and \cite{rene-wang-strong}. Now, if $\process{M}$ is a Markov process such that its corresponding semigroup satisfies the strong Feller property, then the function $x\longmapsto\mathbb{P}^{x}(T_c=\infty)$ is continuous.  In particular, $\process{M}$
satisfies the lower and upper semicontinuity assumptions from Proposition \ref{p3.1}.
Indeed, let $t>0$ be arbitrary. Then, for any $x\in\R^{d}$, by the Markov and strong Feller properties,  we have \begin{align*}\lim_{y\longrightarrow x}\mathbb{P}^{y}(T_c<\infty)&=\lim_{y\longrightarrow x}\mathbb{P}^{y}\left(\inf\{s\geq t:M_s\notin\R^{d}\}<\infty\right)\\&=
\lim_{y\longrightarrow x}\left(\mathbb{E}^{y}\left[\mathbb{P}^{M_t}(T_c<\infty)1_{\R^{d}}(M_t)\right]+\mathbb{P}^{y}(M_t=\infty)\right)\\&
=\lim_{y\longrightarrow x}\mathbb{E}^{y}\left[\mathbb{P}^{M_t}(T_c<\infty)\right]\\&=\mathbb{E}^{x}\left[\mathbb{P}^{M_t}(T_c<\infty)\right]\\&=\mathbb{P}^{x}\left(\inf\{s\geq t:M_s\notin\R^{d}\}<\infty\right)\\&=\mathbb{P}^{x}(T_c<\infty),\end{align*} which proves the assertion.

Note that if $\process{M}$ is irreducible, then it is necessarily conservative (nonexplosive).
Thus, every open-set irreducible L\'evy-type process is always conservative (nonexplosive).
 A sufficient condition for the conservativeness of a L\'evy-type process $\process{F}$ in terms of the corresponding symbol $q(x,\xi)$ (or L\'evy quadruple $(a(x),b(x),c(x),\nu(x,dy))$) is as follows  $$\lim_{k\longrightarrow\infty}\sup_{|y-x|\leq2k}\sup_{|\eta|\leq1/k}|q(y,\eta)|=0,\quad x\in\R^{d},$$ (see \cite[Theorem 5.5]{rene-conserv}). Clearly, the above relation automatically implies that  $a(x)=0$ for all $x\in\R^{d}$. Moreover, in the bounded coefficients case and under the assumption that $a(x)$ is continuous,
$\process{F}$ is conservative if, and only if, $a(x)=0$ for all $x\in\R^{d}$ (see \cite[Theorem 5.2]{rene-conserv}). In the following theorem, under  the assumptions that a   L\'evy-type process (not necessarily with bounded coefficients) is open-set irreducible and the corresponding  function $a(x)$ is lower semicontinuous, we prove that  $a(x)=0$ for all $x\in\R^{d}.$

\begin{theorem}\label{tm3.1}
Let $\process{F}$ be a $d$-dimensional open-set irreducible L\'evy-type process with Feller generator $(\mathcal{A},\mathcal{D}_{\mathcal{A}})$ and L\'evy quadruple $(a(x),b(x),c(x),\nu(x,dy))$. If the function $a(x)$ is lower semicontinuous, then $a(x)=0$ for all $x\in\R^{d}.$
\end{theorem}
\begin{proof}
As we commented above, the irreducibility of  $\process{F}$ automatically implies its conservativeness. Hence, $\mathbb{P}^{x}(F_t\in\R^{d})=1$ for all $x\in\R^{d}$ and $t\geq0$. Next, let  $r>0$ and $R>0$ be fixed and pick some $\varphi_r\in C_c^{2}(\R^{d})$ such that $1_{B(0,r)}(x)\leq \varphi_r(x)\leq 1_{B(0,2r)}(x)$ for all $x\in\R^{d}$.
Here,  $C_c^{k}(\R^{d})$, $k\geq 0$, denotes the space of $k$ times differentiable functions such that all derivatives up to order $k$ have compact support.
According to \cite[Theorem 2.37]{bjoern-rene-jian} (which states that $C_c^{2}(\R^{d})\subseteq\mathcal{D}_{\mathcal{A}}$) and \cite[Theorem 2.2.13 and Proposition 4.1.7]{ethier}, we have
 \begin{align*}\mathbb{E}^{x}\left[\varphi_r(F_{t\wedge\tau_{B^{c}(0,R)}})\right]-\varphi_r(x)&=\mathbb{E}^{x}\left[\int_0^{t\wedge\tau_{B^{c}(0,R)}}\mathcal{A}\varphi_r(F_s)ds\right]\\
&=\mathbb{E}^{x}\left[\int_0^{t\wedge\tau_{B^{c}(0,R)}}\mathcal{L}\varphi_r(F_s)ds\right],\quad x\in\R^{d},\ t\geq0,\end{align*}
  where the operator $\mathcal{L}$ (given by \eqref{eq1.1}) is an extension  of the generator $(\mathcal{A},\mathcal{D}_{\mathcal{A}})$ on $C_b^{2}(\R^{d}).$
Now, by letting $r\longrightarrow\infty$, the dominated convergence theorem entails that
$$0=\mathbb{E}^{x}\left[\int_0^{t\wedge\tau_{B^{c}(0,R)}}\mathcal{L}1_{\R^{d}}(F_s)ds\right]=-\mathbb{E}^{x}\left[\int_0^{t\wedge\tau_{B^{c}(0,R)}}a(F_s)ds\right],\quad x\in\R^{d},\ t\geq0,$$ and, by letting $R\longrightarrow\infty$, the monotone convergence theorem  implies that $$\int_0^{\infty}a(F_t)dt=0,\quad \mathbb{P}^{x}\textrm{-a.s.},\ x\in\R^{d},$$ (recall that $\process{F}$ is conservative and $a(x)\geq0$ for all $x\in\R^{d}$). Thus, if there would exist some $x_0\in\R^{d}$ such that $a(x_0)>0$, then, by the lower semicontinuity of $a(x)$, $a(x)>0$ on some open neighborhood around $x_0$. But this is in contradiction with the open-set irreducibility  of $\process{F}$.
\end{proof}
Note that if $\process{F}$ is an elliptic diffusion  determined by a L\'evy quadruple $(a(x),b(x),c(x),0)$, then, due to \eqref{eq2.1}, the functions $a(x)$, $b(x)$ and $c(x)$ are automatically continuous.
As we have mentioned above,  the conservativeness of a  Markov process is defined thorough the first exit times of open balls.
In the following  theorem we give sufficient conditions for finiteness of exponential moments of these exit times. Let us remark that this result generalizes \cite[Corollary 5.8]{bjoern-rene-jian} where only  finiteness of the first moment has been discussed. First, we prove the following elementary, but very useful, auxiliary result.
\begin{proposition}\label{p3} Let $\process{F}$ be a $d$-dimensional L\'evy-type process with symbol $q_F(x,\xi)$. Then, the process $\process{M}$, $M_t:=(t,F_t)$, $t\geq0$, is a $(d+1)$-dimensional L\'evy-type process with  symbol $q_M((u,x),(\zeta,\xi))=-i\zeta+q_F(x,\xi)$, $(u,x),(\zeta,\xi)\in\R^{d+1}$, $u,\zeta\in\R$.
\end{proposition}
\begin{proof}
Clearly, $\process{M}$ is a $(d+1)$-dimensional Markov process with respect to $\mathbb{P}_M^{(u,x)}(M_t\in B_1\times B_2):=\delta_{u+t}(B_1)\mathbb{P}_F^{x}(F_t\in B_2),$ $(u,x)\in\R^{d+1}$, $t\geq0$, $B_1\in\mathcal{B}(\R)$ and $B_2\in\mathcal{B}(\R^{d})$. Here, $\delta_t(B)$, $t\in\R$, $B\in\mathcal{B}(\R)$, denotes the Dirac delta measure. The Feller and strong continuity properties of $\process{M}$ easily follow from the facts that $C_c(\R^{d+1})$ is dense in $C_\infty(\R^{d+1})$ and $\{\sum_{i=1}^{n}\varphi_i(t)\psi_i(x):n\in\N,\, \varphi_i\in C_c(\R),\, \psi_i\in C_c(\R^{d}),\, i=1,\ldots,n\}$ is dense in $C_c(\R^{d+1})$ (see \cite[Chapter 4.7]{folland}). Finally,  let us denote by   $(\mathcal{A}_F,\mathcal{D}_{\mathcal{A}_F})$ and $(\mathcal{A}_{M},\mathcal{D}_{\mathcal{A}_M})$ the Feller generators of $\process{F}$ and $\process{M}$, respectively. Then, by a straightforward computation (and by employing \cite[Theorem 2.37]{bjoern-rene-jian}), we see that  for all $f\in C^{2}_c(\R^{d+1})$ we have
$$\mathcal{A}_M(f)(u,x)=\frac{\partial f(u,x)}{\partial u}+\mathcal{A}_Ff(u,x), \quad (u,x)\in\R^{d+1}.$$
Hence, $\process{M}$ is a $(d+1)$-dimensional L\'evy-type process with  symbol $q_M((u,x),(\zeta,\xi))=-i\zeta+q_F(x,\xi)$.
\end{proof}

\begin{theorem}\label{tm3.2}Let $\process{F}$ be a $d$-dimensional L\'evy-type process with Feller generator $(\mathcal{A},\mathcal{D}_{\mathcal{A}})$ and symbol $q(x,\xi)$. Assume that for some  $x\in\R^{d}$ and $R>0$ the following two conditions are satisfied
\begin{align}\label{eq3.1}\sup_{|\xi|\leq1/(2R)}\inf_{|y-x|\leq R}{\rm Re}\, q(y,\xi)>0\quad\textrm{and}\quad \sup_{|\xi|\leq1/(2R)}\sup_{|y-x|\leq R}\frac{{\rm Re}\, q(y,\xi)}{|\xi|{\rm Im}\,q(y,\xi)}\geq 2R.
\end{align}
Then, for any $$0<\lambda<\frac{\sqrt{2}}{8}\sup_{|\xi|\leq1/(2R)}\inf_{|y-x|\leq R}{\rm Re}\,q\left(y,\xi\right),$$
  we have $\mathbb{E}^{x}\left[e^{\lambda\tau_{B^{c}(x,R)}}\right]<\infty.$
\end{theorem}
\begin{proof}
Let  $x\in\R^{d}$,  $R>0$  and   $\lambda>0$ be as in the statement of the theorem  and let us put  $\varphi(t):=e^{\lambda t}$, $t\in\R$, and  $\psi(y):=\cos\left(\langle y-x,z\rangle/R\right),$ $y\in\R^{d}$, where $z\in\R^{d}$, $0<|z|\leq1/2$, is such that the first condition in \eqref{eq3.1} is satisfied for $z/R$, that is, $$\inf_{|y-x|\leq R}{\rm Re}\, q\left(y,\frac{z}{R}\right)>0.$$
 Further,  let $a>0$  be fixed and pick some $\varphi_a\in C_c^{2}(\R)$ and $\psi_a\in C_c^{2}(\R^{d})$, such that $1_{B(0,a)}(t)\leq\varphi_a(t)\leq 1_{B(0,2a)}(t)$ for all $t\in\R$ and $1_{B(x,a)}(y)\leq \psi_a(y)\leq 1_{B(x,2a)}(y)$ for all $y\in\R^{d}$. Now, according to Proposition \ref{p3} and \cite[Theorem 2.2.13 and Proposition 4.1.7]{ethier}, we have
\begin{align}\label{eq3.2}&
\mathbb{E}^{x}\left[(\varphi\varphi_a)(u+t\wedge\tau_{B^{c}(x,R)})(\psi\psi_a)(F_{t\wedge\tau_{B^{c}(x,R)}})\right]-(\varphi\varphi_a)(u)(\psi\psi_a)(x)\nonumber\\
&=\mathbb{E}^{x}\left[\int_0^{t\wedge\tau_{B^{c}(x,R)}}((\varphi\varphi_a)'(u+s)(\psi\psi_a)(F_s)+(\varphi\varphi_a)(u+s)\mathcal{A}(\psi\psi_a)(F_s))ds\right]\nonumber\\
&=\mathbb{E}^{x}\left[\int_0^{t\wedge\tau_{B^{c}(x,R)}}((\varphi\varphi_a)'(u+s)(\psi\psi_a)(F_s)+(\varphi\varphi_a)(u+s)\mathcal{L}(\psi\psi_a)(F_s))ds\right],\quad u\in\R,\  t\geq0.\end{align}
Observe that, by letting $a\longrightarrow\infty$, the dominated convergence theorem implies that the above relation  also holds  for $\varphi(t)$ and $\psi(y)$. Next, under \eqref{eq3.1},
 \cite[the proof of Theorem 5.5]{bjoern-rene-jian} shows that for any $y\in\R^{d}$, $|y-x|\leq R$,  we have
$$\mathcal{L}\psi(y)\leq-\frac{\sqrt{2}}{8}\inf_{|y-x|\leq R}{\rm Re}\,q\left(y,\frac{z}{R}\right),$$ which, together with \eqref{eq3.2}, implies
\begin{align*}0&\leq\mathbb{E}^{x}\left[e^{\lambda(u+t\wedge\tau_{B^{c}(x,R)})}\psi(F_{t\wedge\tau_{B^{c}(x,R)}})\right]\nonumber\\&=e^{\lambda u}+\mathbb{E}^{x}\left[\int_0^{t\wedge\tau_{B^{c}(x,R)}}(\lambda e^{\lambda(u+s)}\psi(F_s)+e^{\lambda(u+s)}\mathcal{L}\psi(F_s))ds\right]\\&
\leq e^{\lambda u}+\frac{e^{\lambda u}\left(\lambda-\frac{\sqrt{2}}{8}\inf_{|y-x|\leq R}{\rm Re}\,q\left(y,\frac{z}{R}\right)\right)}{\lambda} \left(\mathbb{E}^{x}\left[e^{\lambda (t\wedge\tau_{B^{c}(x,R)})}\right]-1\right),\quad u\in\R,\ t\geq0.
\end{align*} Finally, by taking $u=0$  and letting $t\longrightarrow\infty$, we get
$$\mathbb{E}^{x}\left[e^{\lambda \tau_{B^{c}(x,R)}}\right]\leq\frac{\frac{\sqrt{2}}{8}\inf_{|y-x|\leq R}{\rm Re}\,q\left(y,\frac{z}{R}\right)}{\frac{\sqrt{2}}{8}\inf_{|y-x|\leq R}{\rm Re}\,q\left(y,\frac{z}{R}\right)-\lambda}.$$
\end{proof}
Let us remark that the second condition in \eqref{eq3.1}  is, for example, satisfied if $$\liminf_{|\xi|\longrightarrow0}\frac{\inf_{|y-x|\leq R}{\rm Re}\, q(y,\xi)}{|\xi|^{\alpha}}>0\quad\textrm{and}\quad \limsup_{|\xi|\longrightarrow0}\frac{\sup_{|y-x|\leq R}{\rm Im}\, q(y,\xi)}{|\xi|}<\infty$$
for some $\alpha\in(0,2).$

\section{Transience and Recurrence}\label{s4}
In this section, we  discuss the recurrence and transience properties of L\'evy-type processes. First, we provide some characterizations of these properties which we need in the sequel (see also \cite[Proposition 2.1]{sandric-TAMS}). Recall that  every Feller semigroup $\{P_t\}_{t\geq0}$
has a unique extension onto the space $B_b(\R^{d})$ (see
\cite[Section 3]{rene-conserv}). For notational simplicity, we
denote this extension again by $\{P_t\}_{t\geq0}$. In particular, if $\process{P}$ is a semigroup of a conservative L\'evy-type process, \cite[Corollary 3.4]{rene-conserv} implies that $\{P_t\}_{t\geq0}$ is also a $C_b$-Feller
semigroup. Now, directly from  \cite[Theorem
4.3]{bjoern-overshoot} and \cite[Theorem
  3.3]{meyn-tweedie-mproc} and \cite[Theorems 4.1, 4.2 and
  7.1]{tweedie-mproc} we get the following.
\begin{proposition}\label{p4.1} Let $\process{F}$ be a $d$-dimensional open-set irreducible  L\'evy-type process. Then, the following  properties are equivalent:
\begin{enumerate}
\item [(i)] $\process{F}$ is recurrent;
\item[(ii)] $\process{F}$ is  Harris
recurrent;
\item[(iii)]  there exists
$x\in\R^{d}$,
$$\mathbb{P}^{x}\left(\liminf_{t\longrightarrow\infty}|F_t-x|=0\right)=1;$$
\item [(iv)]  there exists   a  compact set $K\subseteq\R^{d}$ such that $$\mathbb{P}^{x}(\tau_K<\infty)=1,\quad x\in\R^{d}.$$
\end{enumerate}

\end{proposition}

In addition, if we assume that $\process{F}$ is a strong Feller
process, that is, if the corresponding Feller semigroup also satisfies the strong Feller property, then directly from \cite[Poposition 2.4]{sandric-periodic} we get the following.
\begin{proposition}\label{p4.2} Let $\process{F}$ be a $d$-dimensional open-set irreducible  L\'evy-type process, such that the corresponding Feller semigroup is  a strong Feller semigroup. Then, the following  properties are equivalent:
\begin{enumerate}
\item [(i)] $\process{F}$ is transient;
\item[(ii)]   there exists
$x\in\R^{d}$,
$$\mathbb{P}^{x}\left(\lim_{t\longrightarrow\infty}|F_t|=\infty\right)=1;$$
  \item [(iii)]  there
  exist  $x\in\R^{d}$ and an open bounded set $O\subseteq\R^{d}$, such
  that $$\mathbb{P}^{x}\left(\int_0^{\infty}1_{\{F_t\in
  O\}}dt=\infty\right)=0.$$
\end{enumerate}
\end{proposition}
Let us also remark that in Propositions \ref{p4.1} and \ref{p4.2} we can replace ``there exists $x\in\R^{d}$" with ``for all $x\in\R^{d}$" and ``there exists a compact set $K\subseteq\R^{d}$" with ``for every compact set $K\subseteq\R^{d}$".

Now, we prove the main results of this section.
\begin{proof}[Proof of Theorem \ref{tm1.3} (i)]
Let $\alpha>0$, $x_0>r_0>1$ and $0<\varepsilon\leq1-r_0^{-\alpha}$. According to Proposition \ref{p4.1}, it suffices to prove that $\mathbb{P}^{x}(\tau_{B(0,x_0)}<\infty)<1$ for some $x\in\R^{d}.$
Take $w_\alpha:\R\longrightarrow[0,\infty)$ such that
$w_\alpha\in C^{2}(\R)$, it is symmetric, nondecreasing on $[0,\infty)$, $w_\alpha(0)=1-r_0^{-\alpha}-\varepsilon$ and $w_\alpha(u)=W_\alpha(|u|)$ for $|u|\geq r_0$.
 Define $\bar{W}_\alpha:\R^{d}\longrightarrow[0,\infty)$ by
$\bar{W}_\alpha(x):=w_\alpha(|x|),$ $x\in\R^{d}$.
Clearly, $\bar{W}_\alpha\in C^{2}(\R^{d})$. Next, fix $a>0$ and $R>x_0$  and pick some $\varphi_a\in C_c^{2}(\R^{d})$ such that $1_{B(0,a)}(x)\leq \varphi_a(x)\leq 1_{B(0,2a)}(x)$ for all $x\in\R^{d}$.
Then, due to \cite[Theorem 2.37]{bjoern-rene-jian} and \cite[Theorem 2.2.13 and Proposition 4.1.7]{ethier},
\begin{align*}\mathbb{E}^{x}\left[(\bar{W}_\alpha\varphi_a)(F_{t\wedge  \tau_{B(0,x_0)}\wedge \tau_{B^{c}(0,R)}})\right]-(\bar{W}_\alpha\varphi_a)(x)&= \mathbb{E}^{x}\left[\int_0^{t\wedge \tau_{B(0,x_0)}\wedge \tau_{B^{c}(0,R)}}\mathcal{A}(\bar{W}_\alpha\varphi_a)(F_s)ds\right]
\\&= \mathbb{E}^{x}\left[\int_0^{t\wedge \tau_{B(0,x_0)}\wedge \tau_{B^{c}(0,R)}}\mathcal{L}(\bar{W}_\alpha\varphi_a)(F_s)ds\right]\end{align*} for all $x\in\R^{d}$ and $t\geq0$ (recall that  $C^{2}_c(\R^{d})\subseteq\mathcal{D}_{\mathcal{A}}$ and $\mathcal{A}|_{C^{2}_c(\R^{d})}=\mathcal{L}|_{C^{2}_c(\R^{d})}$). By letting $a\longrightarrow\infty$ and applying the dominated convergence theorem in the previous relation, we get
\begin{align}\label{eq4.6}\mathbb{E}^{x}\left[\bar{W}_\alpha(F_{t\wedge  \tau_{B(0,x_0)}\wedge \tau_{B^{c}(0,R)}})\right]=\bar{W}_\alpha(x)+ \mathbb{E}^{x}\left[\int_0^{t\wedge \tau_{B(0,x_0)}\wedge \tau_{B^{c}(0,R)}}\mathcal{L}\bar{W}_\alpha(F_s)ds\right]\end{align} for all $x\in\R^{d}$ and $t\geq0$. Further, as we have commented in the first section,
$\mathcal{L}\bar{W}_\alpha(x)\geq
T_\alpha(|x|)$ for all $x\in\R^{d}$, $|x|\geq x_0$, where the function $T_\alpha(r)$ is given in \eqref{eq1.4}. Thus, by assumption, $\mathcal{L}\bar{W}_\alpha(x)\geq0$ for all $x\in\R^{d}$, $|x|\geq x_0$.
Now, by using this fact and letting $t\longrightarrow\infty$ in  \eqref{eq4.6}, we get
\begin{align}\label{eq4.7}&W_\alpha(x_0)+\mathbb{P}^{x}(\tau_{B(0,x_0)}> \tau_{B^{c}(0,R)})\nonumber\\&\geq W_\alpha(x_0)\mathbb{P}^{x}(\tau_{B(0,x_0)}\leq \tau_{B^{c}(0,R)})+\mathbb{P}^{x}(\tau_{B(0,x_0)}> \tau_{B^{c}(0,R)})\nonumber\\&\geq
\mathbb{E}^{x}\left[\bar{W}_\alpha(F_{\tau_{B(0,x_0)}})1_{\{\tau_{B(0,x_0)}\leq \tau_{B^{c}(0,R)}\}}\right]+\mathbb{E}^{x}\left[\bar{W}_\alpha(F_{\tau_{B^{c}(0,R)}})1_{\{\tau_{B(0,x_0)}> \tau_{B^{c}(0,R)}\}}\right]\nonumber\\
&=\mathbb{E}^{x}\left[\bar{W}_\alpha(F_{  \tau_{B(0,x_0)}\wedge \tau_{B^{c}(0,R)}})\right]\nonumber\\&\geq \bar{W}_\alpha(x)\end{align} for all $x\in\R^{d}$.
Finally, by letting $R\longrightarrow\infty$, the conservativeness property of $\process{F}$ entails $$\mathbb{P}^{x}(\tau_{B(0,x_0)}=\infty)\geq \bar{W}_\alpha(x)-W_\alpha(x_0),\quad x\in\R^{d}.$$ Thus, due to the fact that $\bar{W}_\alpha(x)=1-|x|^{-\alpha}>1-|x_0|^{-\alpha}=W_\alpha(x_0)$ for all $x\in\R^{d}$, $|x|>|x_0|,$ the assertion follows.
\end{proof}

\begin{proof}[Proof of Theorem \ref{tm1.3} (ii)] Let $\alpha\geq0$ and $x_0>r_0>1$. According to Proposition \ref{p4.1}, it suffices to prove that $\mathbb{P}^{x}(\tau_{B(0,x_0)}<\infty)=1$ for all $x\in\R^{d}.$ We proceed similarly as in the proof of Theorem \ref{tm1.3} (i).
Take $v_\alpha:\R\longrightarrow[0,\infty)$  such that
$v_\alpha\in C^{2}(\R)$, it is symmetric, nondecreasing on $[0,\infty)$ and $v_\alpha(u)=V_\alpha(|u|)$ for $|u|\geq r_0$, and define $\bar{V}_\alpha:\R^{d}\longrightarrow[0,\infty)$ by
$\bar{V}_\alpha(x):=v_\alpha(|x|),$ $x\in\R^{d}$.
Clearly, $\bar{V}_\alpha\in C^{2}(\R^{d})$.
 Next, fix $a>0$ and $R>x_0$  and pick some cut-off function $\varphi_a\in C_c^{2}(\R^{d})$ as in the proof of Theorem \ref{tm1.3} (i).
Similarly as before,
 \cite[Theorem 2.37]{bjoern-rene-jian} and \cite[Theorem 2.2.13 and Proposition 4.1.7]{ethier} imply that
\begin{align*}\mathbb{E}^{x}\left[(\bar{V}_\alpha\varphi_a)(F_{t\wedge  \tau_{B(0,x_0)}\wedge \tau_{B^{c}(0,R)}})\right]-(\bar{V}_\alpha\varphi_a)(x)&= \mathbb{E}^{x}\left[\int_0^{t\wedge \tau_{B(0,x_0)}\wedge \tau_{B^{c}(0,R)}}\mathcal{A}(\bar{V}_\alpha\varphi_a)(F_s)ds\right]\\
&= \mathbb{E}^{x}\left[\int_0^{t\wedge \tau_{B(0,x_0)}\wedge \tau_{B^{c}(0,R)}}\mathcal{L}(\bar{V}_\alpha\varphi_a)(F_s)ds\right]\end{align*} for all  $x\in\R^{d}$ and $t\geq0$. In particular, \begin{align*}\mathbb{E}^{x}\left[(\bar{V}_\alpha\varphi_a)(F_{t\wedge  \tau_{B^{c}(0,R)}})1_{\{\tau_{B(0,x_0)}> \tau_{B^{c}(0,R)}\}}\right]\leq(\bar{V}_\alpha\varphi_a)(x)+\mathbb{E}^{x}\left[\int_0^{t\wedge \tau_{B(0,x_0)}\wedge \tau_{B^{c}(0,R)}}\mathcal{L}(\bar{V}_\alpha\varphi_a)(F_s)ds\right]\end{align*} for all $x\in\R^{d}$ and  $t\geq0$. Again, by  letting $a\longrightarrow\infty$, the dominated and monotone convergence theorems automatically yield
\begin{align}\label{eq4.3}\mathbb{E}^{x}\left[\bar{V}_\alpha(F_{t\wedge  \tau_{B^{c}(0,R)}})1_{\{\tau_{B(0,x_0)}> \tau_{B^{c}(0,R)}\}}\right]\leq \bar{V}_\alpha(x)+\mathbb{E}^{x}\left[\int_0^{t\wedge \tau_{B(0,x_0)}\wedge \tau_{B^{c}(0,R)}}\mathcal{L}\bar{V}_\alpha(F_s)ds\right]\end{align} for all $x\in\R^{d}$ and  $t\geq0$. Note that here we employed the fact that $\mathcal{L}\bar{V}_\alpha(x)$ is locally bounded (assumption \eqref{eq1.5}).
Next, by construction, we have that $\mathcal{L}\bar{V}_\alpha(x)\leq
R_\alpha(|x|)$ for all $x\in\R^{d}$, $|x|\geq x_0$,
 where the function $R_\alpha(r)$ is given in \eqref{eq1.4}.
Hence, by assumption, $\mathcal{L}\bar{V}_\alpha(x)\leq0$ for all $x\in\R^{d}$, $|x|\geq x_0$. Now, by employing this fact and letting $t\longrightarrow\infty$ in the relation in \eqref{eq4.3}, Fatou's lemma implies
\begin{align}\label{eq4.4}V_\alpha (R)\mathbb{P}^{x}(\tau_{B(0,x_0)}> \tau_{B^{c}(0,R)})\leq\mathbb{E}^{x}\left[\bar{V}_\alpha(F{ \tau_{B^{c}(0,R)}})1_{\{\tau_{B(0,x_0)}> \tau_{B^{c}(0,R)}\}}\right]\leq \bar{V}_\alpha(x),\quad x\in\R^{d}.\end{align}
 Finally,
 by letting $R\longrightarrow\infty$, the conservativeness property of $\process{F}$  entails that $$\mathbb{P}^{x}(\tau_{B(0,x_0)}=\infty)=0,\quad x\in\R^{d},$$ which proves the desired result.
 \end{proof}

\section{Ergodicity}\label{s5}
In this section, we discuss  ergodicity properties of L\'evy-type processes. Let $\process{M}$
 be a $d$-dimensional  Markov process.
It is well known that if $\process{M}$ is
 recurrent, then it possesses  a unique (up to constant
multiples) invariant measure $\pi(dx)$ (see \cite[Theorem 2.6]{tweedie-mproc}). If the invariant measure is
finite, then it may be normalized to a probability measure.  If
$\process{M}$ is  recurrent with finite invariant measure, then $\process{M}$ is called \emph{positive
recurrent}, otherwise it is called \emph{null  recurrent}. One would expect
that every positive  recurrent process  is (strongly) ergodic,
but in general this is not true (see \cite{meyn-tweedie-II} and \cite{turbul}). However, in the
case of open-set irreducible  L\'evy-type processes which have  an irreducible skeleton chain, due to Proposition \ref{p4.1} and \cite[Theorem 6.1]{meyn-tweedie-II},
these three properties coincide.

Further, note that  a transient Markov process cannot have a finite invariant measure. Indeed, let
$\process{M}$ be a $d$-dimensional transient Markov  process with finite  invariant  measure $\pi(dx)$.  Then, because of the transience, there
exists a countable
                      covering of $\R^{d}$ with  sets
$\{B_j\}_{j\in\N}\subseteq\mathcal{B}(\R^{d})$, such that for each
$j\in\N$ there is a finite constant $M_j\geq0$ such that
$\int_0^{\infty}\mathbb{P}^{x}(M_t\in B_j)dt\leq M_j$ holds for all $x\in\R^{d}$.
Fix some $t>0$. Then, for each $j\in\N$, we have
$$t\pi(B_j)=\int_0^{t}\int_{\R^{d}}\mathbb{P}^{x}(M_s\in B_j)\pi(dx)ds\leq M_j\pi(\R^{d}).$$ Now, by
letting $t\longrightarrow\infty$ we get that $\pi(B_j)=0$ for all
$j\in\N$, which is impossible.
Therefore,  open-set irreducible transient L\'evy-type processes can  only have infinite invariant measures. Examples of such processes can be found in the class of L\'evy processes. Recall that Lebesgue measure is invariant for every L\'evy process.  Furthermore, since a (non-trivial) L\'evy process  cannot have  finite invariant measure (see \cite[Exercise 29.6]{sato-book}), recurrent L\'evy processes can  only be null recurrent.
In the following theorem we give a sufficient condition for null recurrence of  open-set irreducible L\'evy-type processes.
\begin{theorem}\label{tm5.2}
Let $\process{F}$ be a $d$-dimensional open-set irreducible L\'evy-type process with Feller generator $(\mathcal{A},\mathcal{D}_{\mathcal{A}})$.
Then, $\process{F}$ is null recurrent if there exist  $\alpha_1>0$, $\alpha_2\geq0$, $\beta>0$, $x_0>r_0>1$ and $0<\varepsilon\leq1-r_0^{-\alpha_1}$, such that $T_{\alpha_1}(x)\geq-\beta/V_{\alpha_2}(R)$ holds for all $R>x_0$ and $x\in\R^{d}$, $x_0\leq|x|\leq R$, and \eqref{eq1.5}  and $R_{\alpha_2}(x)\leq0$ hold for all $x\in\R^{d}$, $|x|\geq x_0$, where the functions $T_{\alpha_1}(x)$, $V_{\alpha_2}(r)$ and $R_{\alpha_2}(x)$  are defined in \eqref{eq1.4}.
\end{theorem}
\begin{proof}
Let $\alpha_1>0$, $\alpha_2\geq0$, $\beta>0$, $x_0>r_0>1$ and $0<\varepsilon\leq1-r_0^{-\alpha_1}$. Clearly,  due to Theorem \ref{tm1.3} (ii),  $\process{F}$ is recurrent. Hence, according to \cite[Theorem 4.1]{stramer-tweedie-stability}, in order to prove null recurrence of $\process{F}$, it suffices to prove that  \begin{align}\label{eq5.1}\lambda\left(\left\{x\in\R^{d}:\mathbb{E}^{x}\left[\tau_{B(0,x_0)}\right]=\infty\right\}\right)>0.\end{align}
Let
$\bar{W}_{\alpha_1}:\R^{d}\longrightarrow[0,\infty)$ and $\bar{V}_{\alpha_2}:\R^{d}\longrightarrow[0,\infty)$ be as in the proofs of Theorem \ref{tm1.3} (i) and (ii), that is, $\bar{W}_{\alpha_1}(x):=w_{\alpha_1}(|x|),$ $x\in\R^{d}$, where
 $w_{\alpha_1}:\R\longrightarrow[0,\infty)$ is such that
$w_{\alpha_1}\in C^{2}(\R)$, it is symmetric, nondecreasing on $[0,\infty)$, $w_{\alpha_1}(0)=1-r_0^{-\alpha_1}-\varepsilon$ and $w_{\alpha_1}(u)=1-|u|^{-\alpha_1}$ for $|u|\geq r_0$, and $\bar{V}_{\alpha_2}(x):=v_{\alpha_2}(|x|),$ $x\in\R^{d}$, where
 $v_{\alpha_2}:\R\longrightarrow[0,\infty)$ is such that
$v_{\alpha_2}\in C^{2}(\R)$, it is symmetric, nondecreasing on $[0,\infty)$ and $v_{\alpha_2}(u)=V_{\alpha_2}(|u|)$ for $|u|\geq r_0.$
Now, by assumption, we have that $\mathcal{L}\bar{W}_{\alpha_1}(x)\geq T_{\alpha_1}(x)\geq-\beta/V_{\alpha_2}(R)$ for all $R>x_0$ and $x\in\R^{d}$, $x_0\leq|x|\leq R,$ and
 $\mathcal{L}\bar{V}_{\alpha_2}(x)\leq R_{\alpha_2}(x)\leq0$ for all $x\in\R^{d}$, $|x|\geq x_0$.
Combining these facts with \eqref{eq4.6} and \eqref{eq4.4} we get
\begin{align*}V_{\alpha_2}(R)\,\mathbb{P}^{x}(\tau_{B(0,x_0)}> \tau_{B^{c}(0,R)})\leq \bar{V}_{\alpha_2}(x),\quad x\in\R^{d},\end{align*} and
\begin{align*}\bar{W}_{\alpha_1}(x) -\frac{\beta}{V_{\alpha_2}(R)}\mathbb{E}^{x}\left[ \tau_{B(0,x_0)}\wedge \tau_{B^{c}(0,R)}\right]\leq\mathbb{E}^{x}\left[\bar{W}_{\alpha_1}(F_{\tau_{B(0,x_0)}\wedge \tau_{B^{c}(0,R)}})\right],\quad x\in\R^{d},\end{align*}  respectively.
Further, from \eqref{eq4.7} we see
\begin{align*}\mathbb{E}^{x}\left[\bar{W}_{\alpha_1}(F_{  \tau_{B(0,x_0)}\wedge \tau_{B^{c}(0,R)}})\right]\leq W_{\alpha_1}(x_0)+\mathbb{P}^{x}(\tau_{B(0,x_0)}> \tau_{B^{c}(0,R)}),\quad x\in\R^{d}.\end{align*}
Thus,
$$V_{\alpha_2}(R)\,(\bar{W}_{\alpha_1}(x)-W_{\alpha_1}(x_0))-\bar{V}_{\alpha_2}(x)\leq\beta\,\mathbb{E}^{x}\left[ \tau_{B(0,x_0)}\wedge \tau_{B^{c}(0,R)}\right]\leq\beta\,\mathbb{E}^{x}\left[ \tau_{B(0,x_0)}\right],\quad x\in\R^{d}.$$
Finally, by letting $R\longrightarrow\infty$, we get
$ \mathbb{E}^{x}\left[ \tau_{B(0,x_0)}\right]=\infty$ for all $x\in\R^{d}$, $|x|>x_0,$ which, together with  open-set irreducibility of $\process{F}$, concludes the proof.
\end{proof}

Finally, we prove Theorem \ref{tm1.3} (iii), (iv) and (v).
\begin{proof}[Proof of Theorem \ref{tm1.3} (iii)]
 Let $\alpha\geq0$, $\beta>0$ and $x_0>r_0>1$. According to our previous comment (that is, positive recurrence and (strong) ergodicity are equivalent for open-set irreducible L\'evy-type processes) and \cite[Theorem 4.4]{meyn-tweedie-II}, in order to prove the ergodicity of  $\process{F}$,  it suffices to show that \begin{align}\label{eq5.2}\sup_{x\in B(0,x_0)}\mathbb{E}^{x}\left[\tau^{t_0}_{B(0,x_0)}\right]<\infty\end{align} for some $t_0>0,$ where $\tau^{t_0}_{B(0,x_0)}:=\inf\{t\geq t_0:F_t\in B(0,x_0)\}.$ The proof proceeds similarly as in the case of recurrence. Let   $\bar{V}_\alpha:\R^{d}\longrightarrow[0,\infty)$ be as in the proof of Theorem \ref{tm1.3} (ii). Next,
 fix $a>0$ and $R>x_0$   and pick some cut-off function $\varphi_a\in C_c^{2}(\R^{d})$.
As before, by \cite[Theorem 2.37]{bjoern-rene-jian} and \cite[Theorem 2.2.13 and Proposition 4.1.7]{ethier},
\begin{align*}\mathbb{E}^{x}\left[(\bar{V}_\alpha\varphi_a)(F_{t\wedge  \tau_{B(0,x_0)}\wedge \tau_{B^{c}(0,R)}})\right]&= (\bar{V}_\alpha\varphi_a)(x)+\mathbb{E}^{x}\left[\int_0^{t\wedge \tau_{B(0,x_0)}\wedge \tau_{B^{c}(0,R)}}\mathcal{A}(\bar{V}_\alpha\varphi_a)(F_s)ds\right]\\
&= (\bar{V}_\alpha\varphi_a)(x)+\mathbb{E}^{x}\left[\int_0^{t\wedge \tau_{B(0,x_0)}\wedge \tau_{B^{c}(0,R)}}\mathcal{L}(\bar{V}_\alpha\varphi_a)(F_s)ds\right]
\end{align*} for all $x\in\R^{d}$ and $t\geq0.$ In particular,
\begin{align*} \mathbb{E}^{x}\left[\int_0^{t\wedge \tau_{B(0,x_0)}\wedge \tau_{B^{c}(0,R)}}\mathcal{L}(\bar{V}_\alpha\varphi_a)(F_s)ds\right]+(\bar{V}_\alpha\varphi_a)(x)\geq0,\quad x\in\R^{d},\ t\geq0.
\end{align*}
 Further, by assumption,  $\mathcal{L}\bar{V}_\alpha(x)\leq R_\alpha(x)\leq -\beta$ for all $x\in\R^{d}$, $|x|\geq x_0$. By using this fact, local boundedness  of $\mathcal{L}\bar{V}_\alpha(x)$ (assumption \eqref{eq1.5}) and letting
 $a\longrightarrow\infty$, $t\longrightarrow\infty$ and $R\longrightarrow\infty$ in the above relation, respectively,  the dominated and monotone convergence theorems yield
\begin{align*}\mathbb{E}^{x}\left[ \tau_{B(0,x_0)}\right]\leq\frac{\bar{V}_\alpha(x)}{\beta},\quad x\in\R^{d}.\end{align*}
Now, we prove \eqref{eq5.2}. Let $t_0>0$ be arbitrary. By the Markov property we have
$$\mathbb{E}^{x}\left[\tau^{t_0}_{B(0,x_0)}\right]=\mathbb{E}^{x}\left[\mathbb{E}^{x}\left[\tau^{t_0}_{B(0,x_0)}|\mathcal{F}_{t_0}\right]\right]=t_0+\mathbb{E}^{x}\left[\mathbb{E}^{F_{t_{0}}}\left[\tau_{B(0,x_0)}\right]\right]\leq t_0+\frac{\mathbb{E}^{x}\left[\bar{V}_\alpha(F_{t_0})\right]}{\beta},\quad x\in\R^{d}.$$ Thus, in order to prove \eqref{eq5.2}, it suffices to prove
that $\sup_{x\in B(0,x_0)}\mathbb{E}^{x}\left[\bar{V}_\alpha(F_{t_0})\right]<\infty.$
Again, fix  $a>0$ and $R>x_0$ and pick some cut-off function $\varphi_a\in C_c^{2}(\R^{d})$. As above,
\begin{align*}\mathbb{E}^{x}\left[(\bar{V}_\alpha\varphi_a)(F_{t_0\wedge  \tau_{B^{c}(0,R)}})\right]&=(\bar{V}_\alpha\varphi_a)(x) +\mathbb{E}^{x}\left[\int_0^{t_0\wedge \tau_{B^{c}(0,R)}}\mathcal{A}(\bar{V}_\alpha\varphi_a)(F_s)ds\right]\\
&=(\bar{V}_\alpha\varphi_a)(x) +\mathbb{E}^{x}\left[\int_0^{t_0\wedge \tau_{B^{c}(0,R)}}\mathcal{L}(\bar{V}_\alpha\varphi_a)(F_s)ds\right],\quad x\in\R^{d},\ t\geq0.\end{align*}
Now, by letting $a\longrightarrow\infty$, the local boundedness of $\mathcal{L}\bar{V}_\alpha(x)$ and  dominated and monotone convergence theorems imply that for all  $x\in\R^{d}$ we have
\begin{align}\label{eq5.4}&\mathbb{E}^{x}\left[\bar{V}_\alpha(F_{t_0\wedge  \tau_{B^{c}(0,R)}})\right]\nonumber\\&=\bar{V}_\alpha(x) +\mathbb{E}^{x}\left[\int_0^{t_0\wedge \tau_{B^{c}(0,R)}}\mathcal{L}\bar{V}_\alpha(F_s)ds\right]\nonumber\\
&=\bar{V}_\alpha(x) + \mathbb{E}^{x}\left[\int_0^{t_0\wedge \tau_{B^{c}(0,R)}}1_{B(0,x_0)}(F_s)\mathcal{L}\bar{V}_\alpha(F_s)ds\right]+\mathbb{E}^{x}\left[\int_0^{t_0\wedge \tau_{B^{c}(0,R)}}1_{B^{c}(0,x_0)}(F_s)\mathcal{L}\bar{V}_\alpha(F_s)ds\right]\nonumber\\&\leq
\bar{V}_\alpha(x)+t_0\sup_{x\in B(0,x_0)}|\mathcal{L}\bar{V}_\alpha(x)| +\mathbb{E}^{x}\left[\int_0^{t_0\wedge \tau_{B^{c}(0,R)}}1_{B^{c}(0,x_0)}(F_s)\mathcal{L}\bar{V}_\alpha(F_s)ds\right].\end{align}
In particular, since $\mathcal{L}\bar{V}_\alpha(x)\leq -\beta$ for all $x\in\R^{d}$, $|x|\geq x_0$,
$$\mathbb{E}^{x}\left[\bar{V}_\alpha(F_{t_0\wedge  \tau_{B^{c}(0,R)}})\right]\leq\bar{V}_\alpha(x)+t_0\sup_{x\in B(0,x_0)}|\mathcal{L}\bar{V}_\alpha(x)|,\quad x\in\R^{d}.$$
Finally, by letting $R\longrightarrow\infty$, Fatou's lemma and the conservativeness property of $\process{F}$ imply
\begin{align}\label{eq5.5}\mathbb{E}^{x}\left[\bar{V}_\alpha(F_{t_0})\right]\leq\bar{V}_\alpha(x)+t_0\sup_{x\in B(0,x_0)}|\mathcal{L}\bar{V}_\alpha(x)|,\quad x\in\R^{d},\end{align} that is,
$$\sup_{x\in B(0,x_0)}\mathbb{E}^{x}\left[\tau^{t_0}_{B(0,x_0)}\right]\leq t_0+\frac{V_\alpha(x_0)}{\beta}+\frac{t_0}{\beta}\sup_{x\in B(0,x_0)}|\mathcal{A}\bar{V}_\alpha(x)|,$$ which proves the assertion.
 \end{proof}

\begin{proof}[Proof of Theorem \ref{tm1.3} (iv)]
 Let $\alpha\geq0$, $0<\beta<1$, $\gamma>0$ and $x_0>r_0>1$. First, note that, according to Theorem \ref{tm1.3} (iv),  $\process{F}$ is automatically (strongly) ergodic.    Therefore, in order to prove the polynomial ergodicity of $\process{F}$ with rate of convergence  $t^{\beta/(1-\beta)}$,    according to \cite[Theorem 1]{subgeometric}, it suffices to prove that \begin{align}\label{eq5}\sup_{x\in B(0,x_0)}\mathbb{E}^{x}\left[\left(\tau^{t_0}_{B(0,x_0)}\right)^{1/(1-\beta)}\right]<\infty\quad\textrm{and}\quad \mathbb{E}^{x}\left[\left(\tau^{t_0}_{B(0,x_0)}\right)^{1/(1-\beta)}\right]<\infty,\quad x\in\R^{d},\end{align} for some  $t_0>0,$ where  $\tau^{t_0}_{B(0,x_0)}$ is as in the proof of Theorem \ref{tm1.3} (iii).
Take $v_\alpha:\R\longrightarrow(0,\infty)$  such that
$v_\alpha\in C^{2}(\R)$, it is symmetric, nondecreasing on $[0,\infty)$ and $v_\alpha(u)=V_\alpha(|u|)$ for $|u|\geq r_0$, and define $\bar{V}_\alpha:\R^{d}\longrightarrow[0,\infty)$ by
$\bar{V}_\alpha(x):=v_\alpha(|x|),$ $x\in\R^{d}$.
Clearly, $\bar{V}_\alpha(0)>0$ and $\bar{V}_\alpha\in C^{2}(\R^{d})$.
Next, fix some $\lambda>0$ and define $f(u,x):=(\lambda u+\bar{V}^{1-\beta}_\alpha(x))^{1/(1-\beta)}$. Obviously, $f\in C^{2}([0,\infty)\times\R^{d})$. Now,
by fixing $a>0$ and $R>x_0$  and picking some $\varphi_a\in C_c^{2}(\R^{d+1})$ such that $1_{B(0,a)}(u,x)\leq \varphi_a(u,x)\leq 1_{B(0,2a)}(u,x)$ for all $(u,x)\in\R^{d+1}$, from Proposition \ref{p3}  and \cite[Theorem 2.2.13 and Proposition 4.1.7]{ethier}, we have
\begin{align*}&\mathbb{E}^{x}\left[(f\varphi_a)(u+t\wedge\tau_{B(0,x_0)}\wedge\tau_{B^{c}(0,R)},F_{t\wedge\tau_{B(0,x_0)}\wedge\tau_{B^{c}(0,R)}})\right]-(f\varphi_a)(u,x)\\
&=\mathbb{E}^{x}\left[\int_0^{t\wedge\tau_{B(0,x_0)}\wedge\tau_{B^{c}(0,R)}}\left(\frac{\partial (f\varphi_a)(u+s,F_s)}{\partial u}+\mathcal{A}(f\varphi_a)(u+s,F_s)\right)ds\right]\nonumber\\&=\mathbb{E}^{x}\left[\int_0^{t\wedge\tau_{B(0,x_0)}\wedge\tau_{B^{c}(0,R)}}\left(\frac{\partial (f\varphi_a)(u+s,F_s)}{\partial u}+\mathcal{L}(f\varphi_a)(u+s,F_s)\right)ds\right]\nonumber,\quad x\in\R^{d},\  u,t\geq0.\end{align*}
 Observe that, due to \eqref{eq1.5}, the function $$(u,x)\longmapsto\int_{\{|y|\geq1,\,|y+x|\geq r_0\}}f(u,y+x)dy$$ is locally bounded, hence, by letting $a\longrightarrow\infty$, the dominated convergence theorem and \eqref{eq5.5} yield
\begin{align}\label{eq5.6}& \mathbb{E}^{x}\left[f(u+t\wedge\tau_{B(0,x_0)}\wedge\tau_{B^{c}(0,R)},F_{t\wedge\tau_{B(0,x_0)}\wedge\tau_{B^{c}(0,R)}})\right]-f(u,x)\nonumber\\&=
\mathbb{E}^{x}\left[\int_0^{t\wedge\tau_{B(0,x_0)}\wedge\tau_{B^{c}(0,R)}}\left(\frac{\partial f(u+s,F_s)}{\partial u}+\mathcal{L}f(u+s,F_s)\right)ds\right],\quad x\in\R^{d},\ u,t\geq0.\end{align} Now, let us discuss the right-hand side of \eqref{eq5.6}.
First,  by assumption, it holds that
\begin{align}\label{eq5.7}\mathcal{L}\bar{V}_\alpha(x)\leq R_\alpha(x)\leq-\gamma \bar{V}^{\beta}_\alpha(x)\end{align} for all $x\in\R^{d}$, $|x|\geq x_0.$
Further, note that the process $$\left\{\bar{V}_\alpha(F_{t\wedge\tau_{B(0,x_0)}\wedge\tau_{B^{c}(0,R)}})+\gamma\int^{t\wedge\tau_{B(0,x_0)}\wedge\tau_{B^{c}(0,R)}}_0\bar{V}^{\beta}_\alpha(F_s)ds\right\}_{t\geq0}$$ is a $\mathbb{P}^{x}$-supermartingale for all $x\in\R^{d}.$
Indeed,  \cite[Theorem 2.2.13 and Proposition 4.1.7]{ethier}, we have
\begin{align*}&\mathbb{E}^{x}\left[(\bar{V}_\alpha\varphi_a)(F_{t\wedge\tau_{B(0,x_0)}\wedge\tau_{B^{c}(0,R)}})+\gamma\int^{t\wedge\tau_{B(0,x_0)}\wedge\tau_{B^{c}(0,R)}}_0(\bar{V}^{\beta}_\alpha\varphi_a)(F_v)dv{\Big|}\mathcal{F}_s\right]\\
&=\mathbb{E}^{x}\Bigg[(\bar{V}_\alpha\varphi_a)(F_{t\wedge\tau_{B(0,x_0)}\wedge\tau_{B^{c}(0,R)}})\\
&\hspace{1.2cm} +\gamma\int^{t\wedge\tau_{B(0,x_0)}\wedge\tau_{B^{c}(0,R)}}_0\left((\bar{V}^{\beta}_\alpha\varphi_a)(F_v)+\frac{\mathcal{A}(\bar{V}_\alpha\varphi_a)(F_v)}{\gamma}-\frac{\mathcal{A}(\bar{V}_\alpha\varphi_a)(F_v)}{\gamma}\right)dv{\Big|}\mathcal{F}_s\Bigg]\\
&=(\bar{V}_\alpha\varphi_a)(F_{s\wedge\tau_{B(0,x_0)}\wedge\tau_{B^{c}(0,R)}})- \int^{s\wedge\tau_{B(0,x_0)}\wedge\tau_{B^{c}(0,R)}}_0\mathcal{A}(\bar{V}_\alpha\varphi_a)(F_v)dv\\&\ \ \ +\gamma\int^{s\wedge\tau_{B(0,x_0)}\wedge\tau_{B^{c}(0,R)}}_0\left((\bar{V}^{\beta}_\alpha\varphi_a)(F_v)+\frac{\mathcal{A}(\bar{V}_\alpha\varphi_a)(F_v)}{\gamma}\right)dv\\
&\ \ \ + \gamma\,\mathbb{E}^{x}\left[\int^{t\wedge\tau_{B(0,x_0)}\wedge\tau_{B^{c}(0,R)}}_{s\wedge\tau_{B(0,x_0)}\wedge\tau_{B^{c}(0,R)}}\left((\bar{V}^{\beta}_\alpha\varphi_a)(F_v)+\frac{\mathcal{A}(\bar{V}_\alpha\varphi_a)(F_v)}{\gamma}\right)dv{\Big|}\mathcal{F}_s\right]\\
&=(\bar{V}_\alpha\varphi_a)(F_{s\wedge\tau_{B(0,x_0)}\wedge\tau_{B^{c}(0,R)}})- \int^{s\wedge\tau_{B(0,x_0)}\wedge\tau_{B^{c}(0,R)}}_0\mathcal{L}(\bar{V}_\alpha\varphi_a)(F_v)dv\\&\ \ \ +\gamma\int^{s\wedge\tau_{B(0,x_0)}\wedge\tau_{B^{c}(0,R)}}_0\left((\bar{V}^{\beta}_\alpha\varphi_a)(F_v)+\frac{\mathcal{L}(\bar{V}_\alpha\varphi_a)(F_v)}{\gamma}\right)dv\\
&\ \ \ + \gamma\,\mathbb{E}^{x}\left[\int^{t\wedge\tau_{B(0,x_0)}\wedge\tau_{B^{c}(0,R)}}_{s\wedge\tau_{B(0,x_0)}\wedge\tau_{B^{c}(0,R)}}\left((\bar{V}^{\beta}_\alpha\varphi_a)(F_v)+\frac{\mathcal{L}(\bar{V}_\alpha\varphi_a)(F_v)}{\gamma}\right)dv{\Big|}\mathcal{F}_s\right],
\end{align*}
for all $x\in\R^{d}$ and $u,s,t\geq0$, $s\leq t$. Here, again $a>0$ and
$\varphi_a\in C_c^{2}(\R^{d})$ is such that $1_{B(0,a)}(x)\leq \varphi_a(x)\leq 1_{B(2a)}(x)$ for all $x\in\R^{d}.$
Now, by letting $a\longrightarrow\infty$, \eqref{eq5.5} (which ensures the integrability of the process), the dominated convergence theorem and \eqref{eq5.7} imply
\begin{align*}&\mathbb{E}^{x}\left[\bar{V}_\alpha(F_{t\wedge\tau_{B(0,x_0)}\wedge\tau_{B^{c}(0,R)}})+\gamma\int^{t\wedge\tau_{B(0,x_0)}\wedge\tau_{B^{c}(0,R)}}_0\bar{V}^{\beta}_\alpha(F_v)dv{\Big|}\mathcal{F}_s\right]\\
&=\bar{V}_\alpha(F_{s\wedge\tau_{B(0,x_0)}\wedge\tau_{B^{c}(0,R)}})- \int^{s\wedge\tau_{B(0,x_0)}\wedge\tau_{B^{c}(0,R)}}_0\mathcal{L}\bar{V}_\alpha(F_v)dv\\&\ \ \ +\gamma\int^{s\wedge\tau_{B(0,x_0)}\wedge\tau_{B^{c}(0,R)}}_0\left(\bar{V}^{\beta}_\alpha(F_v)+\frac{\mathcal{L}\bar{V}_\alpha(F_v)}{\gamma}\right)dv\\
&\ \ \ + \gamma\,\mathbb{E}^{x}\left[\int^{t\wedge\tau_{B(0,x_0)}\wedge\tau_{B^{c}(0,R)}}_{s\wedge\tau_{B(0,x_0)}\wedge\tau_{B^{c}(0,R)}}\left(\bar{V}^{\beta}_\alpha(F_v)+\frac{\mathcal{L}\bar{V}_\alpha(F_v)}{\gamma}\right)dv{\Big|}\mathcal{F}_s\right]\\
&\leq\bar{V}_\alpha(F_{s\wedge\tau_{B(0,r_0)}\wedge\tau_{B^{c}(0,R)}})+\gamma\int^{s\wedge\tau_{B(0,r_0)}\wedge\tau_{B^{c}(0,R)}}_0\bar{V}^{\beta}_\alpha(F_v)dv,
\end{align*}
for all $x\in\R^{d}$ and $u,s,t\geq0$, $s\leq t$.
Now, by using this fact, \cite[Corollary 4.5]{hairer}  states that the process
$$\left\{f(u+t\wedge\tau_{B(0,x_0)}\wedge\tau_{B^{c}(0,R)},F_{t\wedge\tau_{B(0,x_0)}\wedge\tau_{B^{c}(0,R)}})-\frac{\lambda-\gamma(1-\beta)}{\lambda}\int_0^{t\wedge\tau_{B(0,x_0)}\wedge\tau_{B^{c}(0,R)}}\frac{\partial f(u+s,F_s)}{\partial u}ds\right\}_{t\geq0}$$ is also a $\mathbb{P}^{x}$-supermartingale for all $x\in\R^{d}$ and $u\geq0$. In particular,
\begin{align}\label{eq5.8}&\mathbb{E}^{x}\left[f(u+t\wedge\tau_{B(0,x_0)}\wedge\tau_{B^{c}(0,R)},F_{t\wedge\tau_{B(0,x_0)}\wedge\tau_{B^{c}(0,R)}})-\frac{\lambda-\gamma(1-\beta)}{\lambda}\int_0^{t\wedge\tau_{B(0,x_0)}\wedge\tau_{B^{c}(0,R)}}\frac{\partial f(u+s,F_s)}{\partial u}ds\right]\nonumber\\&\leq f(u,x)\end{align}
for all $x\in\R^{d}$ and $u,t\geq0$.
 Now, by combining \eqref{eq5.6} and \eqref{eq5.8}, we get
 \begin{align}\label{eq5.9}\mathbb{E}^{x}\left[\int_0^{t\wedge\tau_{B(0,x_0)}\wedge\tau_{B^{c}(0,R)}}\mathcal{L}f(u+s,F_s)ds\right]\leq-\frac{\gamma(1-\beta)}{\lambda}\mathbb{E}^{x}\left[\int_0^{t\wedge\tau_{B(0,x_0)}\wedge\tau_{B^{c}(0,R)}}\frac{\partial f(u+s,F_s)}{\partial u}ds\right]\end{align}
 for all $x\in\R^{d}$ and $u,t\geq0$, and, by combining \eqref{eq5.6} and \eqref{eq5.9}, we obtain
 $$f(u,x)+\frac{\lambda-\gamma(1-\beta)}{\lambda}\mathbb{E}^{x}\left[\int_0^{t\wedge\tau_{B(0,x_0)}\wedge\tau_{B^{c}(0,R)}}\frac{\partial f(u+s,F_s)}{\partial u}ds\right]\geq0,\quad x\in\R^{d},\ u,t\geq0.$$
Specially, by taking $u=0$ and $0<\lambda<\gamma(1-\beta)$, the above relation entails
\begin{align*}\frac{\gamma(1-\beta)-\lambda}{1-\beta}\,\mathbb{E}^{x}\left[\int_0^{t\wedge\tau_{B(0,x_0)}\wedge\tau_{B^{c}(0,R)}}(\lambda s+\bar{V}^{1-\beta}_\alpha(F_s))^{\beta/(1-\beta)}ds\right]\leq \bar{V}_\alpha(x),\quad t\geq0,\ x\in\R^{d}. \end{align*}
By letting $t\longrightarrow\infty$ and $R\longrightarrow\infty$, the monotone convergence theorem and  conservativeness of $\process{F}$ automatically imply
\begin{align*}\frac{\gamma(1-\beta)-\lambda}{1-\beta}\,\mathbb{E}^{x}\left[\int_0^{\tau_{B(0,x_0)}}(\lambda s+\bar{V}^{1-\beta}_\alpha(F_s))^{\beta/(1-\beta)}ds\right]\leq \bar{V}_\alpha(x),\quad x\in\R^{d}. \end{align*}
In particular,
\begin{align*}\mathbb{E}^{x}\left[\tau_{B(0,x_0)}^{1/(1-\beta)}\right]\leq \frac{\lambda^{\beta/(\beta-1)}}{\gamma(1-\beta)-\lambda}\bar{V}_\alpha(x),\quad  x\in\R^{d}. \end{align*}
Now, for an arbitrary $t_0>0$,  the Markov property yields
\begin{align*}\mathbb{E}^{x}\left[\left(\tau^{t_0}_{B(0,x_0)}\right)^{1/(1-\beta)}\right]&=\mathbb{E}^{x}\left[\mathbb{E}^{x}\left[\left(\tau^{t_0}_{B(0,x_0)}\right)^{1/(1-\beta)}{\Big|}\mathcal{F}_{t_0}\right]\right]\\
&=\mathbb{E}^{x}\left[\mathbb{E}^{F_{t_{0}}}\left[\left(t_0+\tau_{B(0,x_0)}\right)^{1/(1-\beta)}\right]\right]\\&
\leq ct_0^{1/(1-\beta)}+c\mathbb{E}^{x}\left[\mathbb{E}^{F_{t_{0}}}\left[\tau_{B(0,x_0)}^{1/(1-\beta)}\right]\right]
\\&
\leq ct_0^{1/(1-\beta)}+
\frac{c\lambda^{\beta/(\beta-1)}}{\gamma(1-\beta)-\lambda}\mathbb{E}^{x}[\bar{V}_\alpha(F_{t_0})],\quad x\in\R^{d},\end{align*}
where in the third step we used the fact that for $a\geq0$, $$(1+t)^{a}\leq c(1+t^{a}),\quad  t\geq0,$$ holds with $c:=\sup_{t\geq0}\frac{(1+t)^{a}}{1+t^{a}}.$ The relations in
\eqref{eq5} now follow by combining the previous result
with \eqref{eq5.5}.
Further, directly from (the proofs of) \cite[Theorem 1]{subgeometric} and \cite[Theorem 4.1]{pekka} we see that
for any $0<\lambda<\gamma(1-\beta)$ and $t_0>0$ there exists $k>0$, such that
   $$\|\mathbb{P}^{x}(X_t\in\cdot)-\pi(\cdot)\|_{TV}\leq k(1-\beta)\mathbb{E}^{x}\left[\left(\tau^{t_0}_{B(0,x_0)}\right)^{1/(1-\beta)}\right]t^{-\beta/(1-\beta)},\quad x\in \R^{d},\  t\geq0,$$ which proves \eqref{eq1.9}.
Finally, to obtain the relation in \eqref{eq1.6} we proceed as follows. First, by combining \eqref{eq5.4} and \eqref{eq5.7} we get $$\mathbb{E}^{x}\left[\bar{V}_\alpha(F_{t\wedge  \tau_{B^{c}(0,R)}})\right]+\gamma\mathbb{E}^{x}\left[\int_0^{t\wedge \tau_{B^{c}(0,R)}}1_{B^{c}(0,x_0)}(F_s)\bar{V}^{\beta}_\alpha(F_s)ds\right]\leq
\bar{V}_\alpha(x)+t\sup_{x\in B(0,x_0)}|\mathcal{L}\bar{V}_\alpha(x)|$$ for all $x\in\R^{d}$ and $t\geq0$. Further, by letting $R\longrightarrow\infty$, Fatou's lemma and the conservativeness of $\process{F}$ entail
$$\mathbb{E}^{x}\left[\bar{V}_\alpha(F_{t})\wedge m\right]+\gamma\mathbb{E}^{x}\left[\int_0^{t}1_{B^{c}(0,x_0)}(F_s)\left(\bar{V}^{\beta}_\alpha(F_s)\wedge m\right)ds\right]\leq
\bar{V}_\alpha(x)+t\sup_{x\in B(0,x_0)}|\mathcal{L}\bar{V}_\alpha(x)|$$ for all $x\in\R^{d}$, $t\geq0$ and $m>0.$ Now, by dividing the above relation by $t$ and letting $t\longrightarrow\infty$, we obtain
$$\gamma\limsup_{t\longrightarrow\infty}\frac{1}{t}\mathbb{E}^{x}\left[\int_0^{t}1_{B^{c}(0,x_0)}(F_s)\left(\bar{V}^{\beta}_\alpha(F_s)\wedge m\right)ds\right]\leq
\sup_{x\in B(0,x_0)}|\mathcal{L}\bar{V}_\alpha(x)|,\quad x\in\R^{d},\ m>0.$$ Finally, by integrating the above relation with respect to $\pi(dx)$ and employing Fatou's lemma and invariance property of $\pi(dx)$, we get
$$\gamma\int_{\R^{d}}1_{B^{c}(0,x_0)}(x)\left(\bar{V}^{\beta}_\alpha(x)\wedge m\right)\pi(dx)\leq
\sup_{x\in B(0,x_0)}|\mathcal{L}\bar{V}_\alpha(x)|,\quad m>0,$$ which, together with Fatou's lemma, proves the assertion.
\end{proof}

\begin{proof}[Proof of Theorem \ref{tm1.3} (v)]
 Let $\alpha\geq0$, $\beta>0$ and $x_0>r_0>1$. Then, again by Theorem \ref{tm1.3} (iii),  $\process{F}$ is (strongly) ergodic.    Therefore, in order to prove the exponential ergodicity of $\process{F}$,  due to \cite[Theorem 6.2]{down}, it suffices to prove that \begin{align}\label{eq5.11}\sup_{x\in B(0,x_0)}\mathbb{E}^{x}\left[e^{\lambda\tau^{t_0}_{B(0,x_0)}}\right]<\infty\quad \textrm{and}\quad \mathbb{E}^{x}\left[e^{\lambda\tau^{t_0}_{B(0,x_0)}}\right]<\infty,\quad x\in\R^{d},\end{align} for some $\lambda>0$ and $t_0>0.$
Let
 $\bar{V}_\alpha:\R^{d}\longrightarrow[0,\infty)$ be as in the proof of Theorem \ref{tm1.3} (ii). By assumption,
$\mathcal{L}\bar{V}_\alpha(x)\leq R_\alpha(|x|)\leq-\beta \bar{V}_\alpha(x)$ for all $x\in\R^{d}$, $|x|\geq x_0.$
Next, fix  $\lambda>0$ and $R>x_0$ and define $\varphi(t):=e^{\lambda t}$ and $\psi(x):=\bar{V}_\alpha(x)$. Now, by similar arguments as in the proof of Theorem \ref{tm3.2} (by applying Proposition \ref{p3} and \cite[Theorem 2.2.13 and Proposition 4.1.7]{ethier}),  we get
\begin{align*}&\mathbb{E}^{x}\left[e^{\lambda(t\wedge\tau_{B(0,x_0)}\wedge\tau_{B^{c}(0,R)})}\bar{V}_\alpha(F_{t\wedge\tau_{B(0,x_0)}\wedge\tau_{B^{c}(x,R)}})\right]\nonumber\\&=\bar{V}_\alpha(x)+\mathbb{E}^{x}\left[\int_0^{t\wedge\tau_{B(0,x_0)}\wedge\tau_{B^{c}(x,R)}}(\lambda e^{\lambda s}\bar{V}_\alpha(F_s)+e^{\lambda s}\mathcal{L}\bar{V}_\alpha(F_s))ds\right],\quad x\in\R^{d},\ t\geq0.\end{align*}
In particular,
\begin{align*}&\bar{V}_\alpha(x)+\mathbb{E}^{x}\left[\int_0^{t\wedge\tau_{B(0,x_0)}\wedge\tau_{B^{c}(x,R)}}(\lambda-\beta) e^{\lambda s}\bar{V}_{\alpha}(F_s)ds\right]\geq0,\quad x\in\R^{d},\ t\geq0.\end{align*}
Thus, by taking $0<\lambda<\beta$, we get
\begin{align*}&\mathbb{E}^{x}\left[\int_0^{t\wedge\tau_{B(0,x_0)}\wedge\tau_{B^{c}(x,R)}} e^{\lambda s}\bar{V}_{\alpha}(F_s)ds\right]\leq\frac{\bar{V}_\alpha(x)}{\beta-\lambda},\quad x\in\R^{d},\ t\geq0,\end{align*}
and, by letting $t\longrightarrow\infty$ and $R\longrightarrow\infty$, Fatou's lemma and the conservativeness property of $\process{F}$ entail
\begin{align*}&\mathbb{E}^{x}\left[\int_0^{\tau_{B(0,x_0)}} e^{\lambda s}\bar{V}_{\alpha}(F_s)ds\right]\leq\frac{\bar{V}_\alpha(x)}{\beta-\lambda},\quad x\in\R^{d}.\end{align*}
Specially, we have
\begin{align*}&\mathbb{E}^{x}\left[ e^{\lambda \tau_{B(0,x_0})}\right]\leq \frac{\lambda}{\beta-\lambda}\bar{V}_\alpha(x)+1,\quad x\in\R^{d}.\end{align*}
Now, for any $t_0>0$,  the Markov property yields
\begin{align}\label{eq5.12}\mathbb{E}^{x}\left[e^{\lambda\tau^{t_0}_{B(0,x_0)}}\right]&=\mathbb{E}^{x}\left[\mathbb{E}^{x}\left[e^{\lambda\tau^{t_0}_{B(0,x_0)}}{\Big|}\mathcal{F}_{t_0}\right]\right]\nonumber\\&=e^{\lambda t_0}\mathbb{E}^{x}\left[\mathbb{E}^{F_{t_{0}}}\left[e^{\lambda\tau_{B(0,x_0)}}\right]\right]\nonumber\\&\leq\frac{\lambda e^{\lambda t_0}}{\beta-\lambda}\mathbb{E}^{x}[\bar{V}_\alpha(F_{t_0})]+e^{\lambda t_0},\quad x\in\R^{d},\end{align}
 which together with \eqref{eq5.5} proves \eqref{eq5.11}. Furthermore, under \eqref{eq5.11},  (the proofs of) \cite[Theorems 5.2 and 6.2]{down} imply that for any $\kappa>0$ there exists $k(\kappa)>0$, such that
$$\|\mathbb{P}^{x}(F_t\in\cdot)-\pi(\cdot)\|_{TV}\leq \left(1+ \mathbb{E}^{x}\left[\int_0^{\tau^{t_0}_{B(0,x_0)}}e^{\lambda t}dt\right]\right)e^{k(\kappa)-\kappa t},\quad x\in\R^{d},\ t\geq0.$$  Thus, by combining this with  \eqref{eq5.5} and \eqref{eq5.12}, we automatically conclude \eqref{eq1.7}.
Finally, the proof of the relation in \eqref{eq1.8} follows by employing completely the same arguments as in the proof of \eqref{eq1.6}, which concludes the proof of the theorem.
\end{proof}

 \section*{Acknowledgement} This work has been supported in part by the Croatian Science Foundation under Project 3526 and  NEWFELPRO Programme  under Project 31. The author thanks the referees for their  helpful comments
and careful corrections.
\bibliographystyle{alpha}
\bibliography{References}

\end{document}